\newif\ifpersonal

\personaltrue 

\ifpersonal
\newcommand*{\personal}[1]{\tcolor[rgb]{0,0,1}{(Personal: #1)}}
\newcommand*{\todo}[1]{\textcolor{red}{(Todo: #1)}}
\else
\newcommand*{\personal}[1]{\ignorespaces}
\newcommand*{\todo}[1]{\ignorespaces}
\fi

\documentclass[oneside, hidelinks, 12pt, a4paper,reqno]{amsart}
\usepackage{extsizes}
\usepackage[toc,page]{appendix}
\usepackage[foot]{amsaddr}
\usepackage{etoolbox}

\makeatletter
\patchcmd{\@setaddresses}{\indent}{\noindent}{}{}
\patchcmd{\@setaddresses}{\indent}{\noindent}{}{}
\patchcmd{\@setaddresses}{\indent}{\noindent}{}{}
\patchcmd{\@setaddresses}{\indent}{\noindent}{}{}
\makeatother
\usepackage{amsthm,mathtools,bm,tensor}
\newcommand*{\boldone}{\text{\usefont{U}{bbold}{m}{n}1}}
\usepackage[T1]{fontenc}
\DeclareFontFamily{U}{min}{}
\DeclareFontShape{U}{min}{m}{n}{<-> dmjhira}{}
\usepackage[bitstream-charter]{mathdesign}
\usepackage{XCharter}
\let\mathcal\undefined
\DeclareMathAlphabet{\mathcal}{U}{dutchcal}{m}{n}
\usepackage{enumerate,comment,braket,xspace,tikz-cd,csquotes}
\usepackage[driver=dvips,centering,vscale=0.7,hscale=0.8]{geometry}
\usepackage{url}
\usepackage{microtype}
\usepackage[greek.polutoniko,english]{babel}

\usepackage[
  style=authoryear,      
  backend=biber,         
  sorting=nyt,
  giveninits=true, 
  sorting=nyt,citestyle=ieee-alphabetic, bibstyle=alphabetic,texencoding=ascii,bibencoding=utf8, 
  maxnames=5, minnames=3, maxbibnames=99
]{biblatex}

\DeclareNameAlias{author}{family-given}
\DeclareNameAlias{editor}{family-given}

\usepackage{setspace}
\usepackage{enumitem}
\usepackage[cal=dutchcal, scr=boondoxo]{mathalfa}
\usepackage{listings}
\usepackage{colordvi}
\usepackage{tikz,pgf}
\usepackage[normalem]{ulem}
\usetikzlibrary{arrows,decorations.pathreplacing,backgrounds,positioning,fit,matrix,calc}
\tikzset{curve/.style={settings={#1},to path={(\tikztostart)
			.. controls ($(\tikztostart)!\pv{pos}!(\tikztotarget)!\pv{height}!270:(\tikztotarget)$)
			and ($(\tikztostart)!1-\pv{pos}!(\tikztotarget)!\pv{height}!270:(\tikztotarget)$)
			.. (\tikztotarget)\tikztonodes}},
	settings/.code={\tikzset{quiver/.cd,#1}
		\def\pv##1{\pgfkeysvalueof{/tikz/quiver/##1}}},
	quiver/.cd,pos/.initial=0.35,height/.initial=0}
\tikzset{tail reversed/.code={\pgfsetarrowsstart{tikzcd to}}}
\tikzset{2tail/.code={\pgfsetarrowsstart{Implies[reversed]}}}
\tikzset{2tail reversed/.code={\pgfsetarrowsstart{Implies}}}
\usepackage{epsfig}
\usepackage{epstopdf}
\usepackage{tikz-cd}
\usepackage{fancyhdr}
\usepackage{graphicx}
\usepackage{fullpage}

\usepackage{comment}
\usepackage{color}
\usepackage[totoc]{idxlayout}
\usepackage[lastexercise]{exercise}
\usepackage{xcolor}
\usepackage{breakcites}
\newcommand{\bDelta}{\bm{\Delta}} 

\makeatletter
\def\@tocline#1#2#3#4#5#6#7{\relax
	\ifnum #1>\c@tocdepth 
	\else
	\par \addpenalty\@secpenalty\addvspace{#2}%
	\begingroup \hyphenpenalty\@M
	\@ifempty{#4}{%
		\@tempdima\csname r@tocindent\number#1\endcsname\relax
	}{%
		\@tempdima#4\relax
	}%
	\parindent\z@ \leftskip#3\relax \advance\leftskip\@tempdima\relax
	\rightskip\@pnumwidth plus4em \parfillskip-\@pnumwidth
	#5\leavevmode\hskip-\@tempdima
	\ifcase #1
	\or\or \hskip 1em \or \hskip 2em \else \hskip 3em \fi%
	#6\nobreak\relax
	\dotfill\hbox to\@pnumwidth{\@tocpagenum{#7}}\par
	\nobreak
	\endgroup
	\fi}
\makeatother

\makeatletter
\renewcommand\paragraph{\@startsection{paragraph}{5}{\z@}%
  {3.25ex \@plus1ex \@minus.2ex}%
  {-1em}%
  {\normalfont\normalsize\bfseries}}
\renewcommand\subparagraph{\@startsection{subparagraph}{6}{\parindent}%
  {3.25ex \@plus1ex \@minus .2ex}%
  {-1em}%
  {\normalfont\normalsize\bfseries}}
\def\toclevel@subsubsubsection{4}
\def\toclevel@paragraph{5}
\def\toclevel@paragraph{6}
\def\l@subsubsubsection{\@dottedtocline{4}{7em}{4em}}
\def\l@paragraph{\@dottedtocline{5}{10em}{5em}}
\def\l@subparagraph{\@dottedtocline{6}{14em}{6em}}
\makeatother

\usepackage{marginnote}

\usepackage{lipsum}
\usepackage{fancyhdr}
\providecommand{\abstract}{}
\usepackage{abstract}
\DeclareGraphicsExtensions{.png,.jpg,.pdf,.mps}
\usepackage{fancyhdr}
\usepackage{lastpage}
\usepackage{multirow}
\usepackage{faktor}

\usepackage{array}
\usepackage{makecell}

\newcommand{\stackspace}{1.7}
\newcommand{\stack}[2][1cm]{\;\tikz[baseline, yshift=.65ex]%
	{\foreach \k [evaluate=\k as \r using (.5*#2+.5-\k)*\stackspace] in {1,...,#2}{%
			\ifodd\k{\draw[->](0,\r pt)--(#1,\r pt);}%
			\else{\draw[<-](0,\r pt)--(#1,\r pt);}\fi
	}}\;}
 \newcommand{\stackrev}[2][1cm]{\;\tikz[baseline, yshift=.65ex]%
	{\foreach \k [evaluate=\k as \r using (.5*#2+.5-\k)*\stackspace] in {1,...,#2}{%
			\ifodd\k{\draw[<-](0,\r pt)--(#1,\r pt);}%
			\else{\draw[->](0,\r pt)--(#1,\r pt);}\fi
	}}\;}

\tikzcdset{
  diagrams={>={Straight Barb[scale=0.5]}}
}
\tikzset{
  altstackar/.style={decorate, decoration={show path construction,
    lineto code={
      \path (\tikzinputsegmentfirst); \pgfgetlastxy{\xstart}{\ystart}
      \path (\tikzinputsegmentlast); \pgfgetlastxy{\xend}{\yend}
      \path ($(0,0)!1.5pt!(\ystart-\yend,\xend-\xstart)$); \pgfgetlastxy{\xperp}{\yperp}
      \foreach \n[evaluate=\n as \k using .5*#1-\n+.5] in {1,...,#1}{
        \ifodd\n{\draw[->, shorten <=2pt, shift={($\k*(\xperp,\yperp)$)}](\xstart,\ystart)--(\xend,\yend);}
        \else{\draw[<-, shorten >=2pt, shift={($\k*(\xperp,\yperp)$)}](\xstart,\ystart)--(\xend,\yend);}\fi
      }
    }
  }}, altstackar/.default={1}
}

\usepackage{pict2e}

\makeatletter
\newcommand{\adjunction}[4]{%
  #1\colon #2%
  \mathrel{\vcenter{%
    \offinterlineskip\m@th
    \ialign{%
      \hfil$##$\hfil\cr
      \longrightharpoonup\cr
      \noalign{\kern-.3ex}
      \smallbot\cr
      \longleftharpoondown\cr
    }%
  }}%
  #3 \noloc #4%
}
\newcommand{\longrightharpoonup}{\relbar\joinrel\rightharpoonup}
\newcommand{\longleftharpoondown}{\leftharpoondown\joinrel\relbar}
\newcommand\noloc{%
  \nobreak
  \mspace{6mu plus 1mu}
  {:}
  \nonscript\mkern-\thinmuskip
  \mathpunct{}
  \mspace{2mu}
}
\newcommand{\smallbot}{%
  \begingroup\setlength\unitlength{.15em}%
  \begin{picture}(1,1)
  \roundcap
  \polyline(0,0)(1,0)
  \polyline(0.5,0)(0.5,1)
  \end{picture}%
  \endgroup
}
\makeatother

\usepackage{hyperref}
\hypersetup{
	colorlinks=true,
	linkcolor=red,
        linktoc=page,
	anchorcolor=black,
	filecolor=blue,
	citecolor = blue, 
	menucolor=., 
	urlcolor=cyan,
	breaklinks=true,
    unicode=true,
}
\usepackage[capitalize]{cleveref}
\usepackage{aliascnt}

\theoremstyle{plain}
\newtheorem{theorem}{Theorem}[section]

\newaliascnt{conjecture}{theorem}
\newtheorem{conjecture}[conjecture]{Conjecture}
\aliascntresetthe{conjecture}

\theoremstyle{definition}
\newaliascnt{definition}{theorem}
\newtheorem{definition}[definition]{Definition}
\aliascntresetthe{definition}

\newaliascnt{example}{theorem}

\aliascntresetthe{example}

\newaliascnt{remark}{theorem}
\newtheorem{remark}[remark]{Remark}
\aliascntresetthe{remark}

\newaliascnt{construction}{theorem}
\newtheorem{construction}[construction]{Construction}
\aliascntresetthe{construction}

\theoremstyle{plain}
\newaliascnt{lemma}{theorem}
\newtheorem{lemma}[lemma]{Lemma}
\aliascntresetthe{lemma}

\newaliascnt{corollary}{theorem}
\newtheorem{corollary}[corollary]{Corollary}
\aliascntresetthe{corollary}

\newaliascnt{proposition}{theorem}
\newtheorem{proposition}[proposition]{Proposition}
\aliascntresetthe{proposition}

\newaliascnt{assumptions}{theorem}
\newtheorem{assumptions}[assumptions]{Assumptions}
\aliascntresetthe{assumptions}

\newaliascnt{setup}{theorem}
\newtheorem{setup}[setup]{Setup}
\aliascntresetthe{setup}

\Crefname{theorem}{Theorem}{Theorems}
\Crefname{conjecture}{Conjecture}{Conjectures}
\Crefname{definition}{Definition}{Definitions}
\Crefname{example}{Example}{Examples}
\Crefname{remark}{Remark}{Remarks}
\Crefname{construction}{Construction}{Constructions}
\Crefname{lemma}{Lemma}{Lemmas}
\Crefname{corollary}{Corollary}{Corollaries}
\Crefname{proposition}{Proposition}{Propositions}
\Crefname{assumptions}{Assumption}{Assumptions}

 \newcommand{\Z}{\mathbb Z}

 \newcommand{\rL}{\mathrm L}

 \newcommand{\cA}{\mathscr A}
 
 \newcommand{\cC}{\mathscr C}
 \newcommand{\cD}{\mathscr D}
 
 \newcommand{\cF}{\mathscr F}
 
 \newcommand{\cG}{\mathscr G}

 \newcommand{\cM}{\mathscr M}
 
 \newcommand{\cO}{\mathscr O}
 \newcommand{\cP}{\mathscr P}

 \newcommand{\cV}{\mathscr V}
 
 \newcommand{\cX}{\mathscr X}

\newcommand{\bbE}{\mathbb E}

\newcommand{\Sch}{\mathrm{Sch}}

\newcommand{\Sp}{\mathrm{Sp}}
\newcommand{\dSp}{\mathrm{dSp}}
\newcommand{\Spec}{\mathrm{Spec}}

\newcommand{\Spa}{\mathrm{Spa}}
\newcommand{\Spk}{\mathrm{Sp}(k)}

\newcommand{\Perf}{\mathrm{Perf}}
\newcommand{\cat}{$\infty$-category }
\newcommand{\cates}{$\infty$-categories }
\newcommand{\catperf}{\mathrm{Cat}^{\Perf}}

\newcommand{\LPr}{\mathrm{Pr}^\rL}
\newcommand{\LPrst}{\mathrm{P}\mathrm{r}^\rL_{\mathrm{st}}}

\newcommand{\Nuc}{\mathrm{Nuc}}
\newcommand{\Qcoh}{\mathrm{QCoh}}
\newcommand{\Fun}{\mathrm{Fun}}
\newcommand{\Hom}{\mathrm{Hom}}
\newcommand{\Mod}{\mathrm{Mod}}
\newcommand{\RMod}{\mathrm{RMod}}
\newcommand{\LMod}{\mathrm{LMod}}

\DeclareMathOperator*{\colim}{colim}
\newcommand{\Lim}{\mathrm{lim}}

\newcommand{\LPrV}{\mathrm{Pr}^{\mathrm{L}}_{\cV}}
\newcommand{\coev}{\mathrm{coev}}
\newcommand{\ev}{\mathrm{ev}}
\newcommand{\Map}{\mathrm{Map}}
\newcommand{\map}{\mathrm{map}}

\newcommand{\Ani}{\mathrm{Ani}}
\newcommand{\otimesV}{\otimes_{\cV}}

\newcommand{\cotimes}{\widehat{\otimes}}

\newcommand{\Solid}{\blacksquare}

\DeclarePairedDelimiter{\tate}{\langle}{\rangle}

\usepackage{bbm}
\usepackage{tensor}

\newcommand{\Addresses}{{
  \bigskip
  \footnotesize

  Matteo Montagnani, \textsc{SISSA, Via Bonomea 265, 34136 Trieste TS, Italy}\par\nopagebreak
  \textit{E-mail address}: \href{mailto:mmontagn@sissa.it}{\texttt{mmontagn@sissa.it}}

  \medskip

}}

\usepackage{doc}

\title{Smooth categories in a 6 functor formalism and compact generation for nuclear categories in analytic geometry}

\author{Matteo Montagnani}

\begin{document}

\maketitle

\section*{Abstract}
In this paper, we study the notion of smooth $\infty$-categories within the framework of a six-functor formalism. By leveraging the theory of condensed mathematics and analytic stacks, we apply these results to demonstrate that a rigid analytic variety is smooth if and only if its associated category of nuclear sheaves is smooth. Furthermore, we relate the compact generation of the category of nuclear sheaves to the algebraization of the rigid analytic variety; these results are then employed to obtain an example of a non atomically generated but internally smooth category.

\tableofcontents

\section{Introduction}
Over the last few decades, algebraic geometry has undergone a profound paradigm shift, evolving toward the intrinsic analysis of varieties via their associated categories. In the modern development of non-commutative algebraic geometry, highly structured categories (such as stable \cates) are treated as geometric spaces. Within this landscape, the precise translation of geometric properties, such as a variety being smooth or proper, into purely categorical definitions over $\Qcoh(X)$ or $\Perf(X)$ has become a foundational tool. A natural ambition is to transport this powerful categorical machinery into the realm of (rigid) analytic geometry.

Another major direction is the comparison between algebraic and analytic spaces, as illustrated by GAGA-style theorems and the algebraization problem: understanding the precise conditions under which a rigid analytic variety X arises as the analytification of a scheme.

However, attempting to reconcile these algebraic and analytic perspectives at the categorical level reveals a fundamental technical obstacle. As explicitly formalized in \Cref{def: 0 introduciton}, the standard notion of smooth and proper \cates crucially relies on the tensor product of stable presentable \cates introduced in \cite{HA}. This tensor product is intrinsically an ``algebraic'' operation and behaves poorly in the analytic setting. For example, the categorical K\"unneth formula fails for Tate algebras over a non-Archimedean field $k$, in the sense that
\begin{equation}\label{eq: kunnet intro}
\Perf(k\langle t \rangle) \otimes_k \Perf(k \langle u \rangle) \neq \Perf(k\langle t,u \rangle).     
\end{equation}

Consequently, the of smooth and proper \cat is essentially algebraic in nature and ill-suited for the analytic framework.

Tate algebras are the fundamental building blocks of rigid analytic geometry. They are utilized to construct rigid analytic varieties in much the same way that affine schemes are used to construct schemes in algebraic geometry (see \cite{bosch2014lectures} for a comprehensive introduction). Throughout this paper, we consider rigid analytic varieties over a complete non-Archimedean field $k$ of characteristic $0$, and we assume all such varieties are locally of finite type. 

Following \cite[Section 5.4]{bosch2014lectures}, it is possible to define the \emph{analytification functor} $(-)^{\mathrm{an}}$ from the category of schemes locally of finite type over $k$ to the category of rigid analytic varieties over $k$. The characterization of the essential image of this functor, obtained in \cite{montagnani2025algebraizationrigidanalyticvarieties}, is a clear manifestation of the fact that the notion of a smooth \cat in \Cref{def: 0 introduciton} is intrinsically algebraic. Indeed, the following theorem was established in \emph{loc.\ cit.}

\begin{theorem}[{\cite[Theorem 3.1]{montagnani2025algebraizationrigidanalyticvarieties}}]\label{thm: intro algebraization via perf}
Let $X$ be a smooth and proper, quasi-compact, and separated rigid analytic variety over $k$. Then $\Perf(X)$ is smooth and proper as a $k$-linear \cat if and only if $X$ is algebraizable (i.e., there exists an algebraic space $Y$ such that $Y^{\mathrm{an}} \simeq X$).
\end{theorem}
See also \cite{toen2007algebraizationcomplexanalyticvarieties} for a closely related theorem in complex analytic geometry.

The objective of this paper is threefold. First, based on the theory of \emph{condensed mathematics} and \emph{analytic stacks} developed by Clausen--Scholze, we introduce a non-commutative criterion to detect the smoothness of rigid analytic varieties; in particular, we obtain a notion of non-commutative smoothness that is well-defined in the analytic setting. Second, we introduce a non-commutative criterion to detect their algebraizability. Finally, we apply these results to construct counterexamples to the atomic generation of smooth enriched categories.


\subsection{Categorical smoothness in the algebraic setting: Motivations}
The investigation of an algebraic variety $X$ via its abelian category of quasi-coherent sheaves, $\Qcoh(X)$, or its subcategory of coherent sheaves, $\mathrm{Coh}(X)$, perfectly captures its essential geometric properties; classical reconstruction theorems (such as those of Rosenberg \cite{brandenburg2013rosenberg}) even demonstrate that these categories can fully recover the underlying space. Nevertheless, this abelian framework ultimately proved too rigid for broader geometric and categorical manipulations.

To obtain a more flexible and homologically robust framework, the focus naturally shifted toward the associated triangulated categories. Following the seminal work of Bondal and Orlov \cite{bondal2001reconstruction}, it became evident that the bounded derived category of coherent sheaves, $\mathrm{D^{b}Coh}(X)$, encodes a wealth of deep geometric information about $X$. However, classical triangulated categories are known to suffer from technical shortcomings, most notably the lack of functorial cones and the absence of well-behaved limits and colimits.

These technical limitations necessitated a further conceptual leap: lifting these triangulated structures to appropriate homotopical enhancements. By employing \emph{differential graded (dg) categories} or, more broadly, \emph{\cates}, one can retain the higher homotopical data that is inevitably lost upon passing to the standard homotopy category. Consequently, the field has increasingly focused on the \cat of quasi-coherent sheaves $\Qcoh(X)$ and the category of perfect complexes $\Perf(X)$ (the $\infty$-categorical generalization of vector bundles).

Building on the pioneering work of Orlov \cite{Orlov_2016noncommutativeschemes} in the triangulated setting, the subsequent evolution of categorical enhancements paved the way for the modern development of \emph{non-commutative algebraic geometry}. In this framework, highly structured categories (such as dg-categories or stable \cates) are no longer viewed merely as invariants of spaces, but are treated as geometric spaces in their own right. Within this landscape, the notion of a \emph{smooth category} (respectively, a \emph{proper category}) emerges as the natural and necessary generalization of a classical smooth (respectively, proper) algebraic variety.

Formally, following the philosophy introduced by Kontsevich \cite{kontsevich1998triangulated}, a dg-category $\mathscr{C}$ is defined as \emph{smooth} if the diagonal bimodule represented by $\mathscr{C}$ in $\mathscr{C}\otimes \mathscr{C}^{\mathrm{op}}$ is a compact object. In the $\infty$-categorical setting, the definitions of smooth and proper \cates can be rephrased as follows:

\begin{definition}\label{def: 0 introduciton}
 Let $\mathscr{C}$ be a stable presentable \cat, and $\mathscr{V}$ a symmetric monoidal, stable, presentable \cat.
\begin{enumerate}
\item We say that $\mathscr{C}$ is dualizable if there exists a stable presentable \cat $\cC^{\vee}$, together with two functors in $\LPrV$, called the evaluation and coevaluation functors:
\begin{equation*}
\mathrm{ev}: \cC \otimesV \cC^{\vee} \to \cV
\end{equation*}
\begin{equation*}
\mathrm{coev}: \cV \to \cC^{\vee} \otimesV \cC,
\end{equation*}
such that the composite maps
\begin{equation}\label{eq: compatibility 1}
\begin{tikzcd}
\cC && {\cC\otimesV \cC^{\vee} \otimesV \cC} && \cC
\arrow["{\mathrm{id}\boxtimes \mathrm{coev}}", from=1-1, to=1-3]
\arrow["{\mathrm{ev} \boxtimes \mathrm{id}}", from=1-3, to=1-5]
\end{tikzcd} 
\end{equation}
\begin{equation}\label{eq: compatibility 2}
\begin{tikzcd}
{\cC^{\vee}} && {\cC^{\vee} \otimesV \cC \otimesV \cC^{\vee}} && {\cC^{\vee}}
\arrow["{\mathrm{coev} \boxtimes \mathrm{id}}", from=1-1, to=1-3]
\arrow["{\mathrm{id} \boxtimes \mathrm{ev}}", from=1-3, to=1-5]
\end{tikzcd}
\end{equation}
are equivalent to the respective identity functors.
\item We say that $\mathscr{C}$ is \emph{smooth} if the coevaluation map admits a right adjoint that preserves colimits.
\item We say that $\mathscr{C}$ is proper if the evaluation map admits a right adjoint that preserves colimits.
\end{enumerate}
\end{definition}

Similarly, if $\mathscr{C}$ is a small, idempotent complete \cat, we say that $\mathscr{C}$ is smooth (respectively, proper) if its ind-completion is smooth (respectively, proper) as in \Cref{def: 0 introduciton}.

When $\mathscr{C}$ is the \cat of quasi-coherent sheaves associated with an algebraic variety $X$, the condition that $\mathscr{C}$ is smooth (respectively, proper) is strictly equivalent to $X$ being a smooth (respectively, proper) variety in the classical sense (see \cite{lunts2009categoricalresolutionsingularities} and \cite{bondal2002generatorsrepresentabilityfunctorscommutative}). We will provide an alternative proof of these known results in a more general setting in \Cref{Appendix 1}. However, the true power of this definition lies in its applicability to purely non-commutative contexts, decoupling the condition of regularity from the existence of an underlying topological space.


Currently, categorical smoothness is not merely a technical artifact, but a foundational requirement driving several major research programs in algebraic geometry. For instance, following the seminal works of Lunts \cite{lunts2009categoricalresolutionsingularities} and Kuznetsov \cite{kuznetsov2015categorical}, the concept of smooth categories allows for the introduction of \emph{categorical resolutions of singularities}. A categorical resolution consists of a fully faithful embedding of the category of perfect complexes, $\Perf(X)$, into a smooth and proper dg-category.

In the realm of \emph{non-commutative algebraic geometry}, the concept of smooth and proper categories is now a central pillar. It facilitates the extension of fundamental algebraic geometric results to the non-commutative setting, while simultaneously presenting new challenges at the interface of the commutative and non-commutative worlds (see, for example, \cite{perry2019noncommutative} and \cite{beraldo2024proofdelignemilnorconjecture}).

Motivated by the ambition to develop a parallel non-commutative theory for rigid analytic spaces, we will work within the setting of \emph{derived rigid analytic geometry}, a homotopical generalization of the classical framework detailed in \cite{bosch2014lectures}. This theory, developed in \Cref{Section 2} and \Cref{Section 3}, is highly inspired by \cite{mikami2023fppfdescentcondensedanimatedrings}. This approach allows us to establish a six-functor formalism and obtain a categorical K\"unneth formula, in stark contrast to the failure observed in the classical analytic setting (\Cref{eq: kunnet intro}). These are the primary ingredients we will utilize to relate the categorical notion of smoothness to its geometric counterpart in the rigid analytic setting.

We emphasize that, because our methods rely on a ``categorical analytic tensor product'' and the existence of a six-functor formalism, they can be easily translated to other settings. For example, they can be adapted to complex analytic geometry or applied to other six-functor formalisms sharing similar features.

\subsection{Categorical smoothness in the analytic setting: A new framework}

Derived analytic geometry was initially developed by Porta and Porta--Yu (see, for example, \cite{porta2020representability,porta2018derivedcomplexanalyticgeometry, porta2018derived, porta2022non}), utilizing an approach similar to the one employed by Lurie in \cite{SAG}. Another approach to derived analytic geometry, closer in style to the framework developed for derived algebraic geometry in \cite{hag}, has recently been introduced and utilized, for instance, in \cite{benbassat2024perspectivefoundationsderivedanalytic, BAMBOZZI20181865}. However, the setting for derived analytic geometry employed in this paper relies on the recent work of Clausen and Scholze on \emph{condensed mathematics} and their theory of \emph{analytic stacks}, as developed in \cite{analytic, clausen2019lectures, Clausen_Scholze_lectures, complex}. Within this framework, one can study rigid analytic varieties, as demonstrated in \cite{mikami2023fppfdescentcondensedanimatedrings, andreychev2021pseudocoherent}, and even construct fundamental new analytic stacks in rigid geometry, as explained in \cite{camargo2024analyticrhamstackrigid}. One of the powerful features of this approach is that it admits a robust six-functor formalism and a ``categorical completed tensor product''.

Throughout this paper, we work over a non-Archimedean field $k$ of characteristic zero. We adopt the definition of \emph{animated affinoid algebras} introduced in \cite[Definition 2.11]{mikami2023fppfdescentcondensedanimatedrings} (see \Cref{def: derived affinoid}). Derived rigid analytic spaces are then defined by ``gluing animated affinoid algebras along analytic open immersions'' (see \Cref{def: derived rigid variety}). We introduce and study the sheaf of nuclear \cates over derived rigid analytic varieties; in \Cref{Section 3}, we will demonstrate its favorable properties.

The notions of \emph{nuclear} objects and \emph{nuclear subcategories} were introduced in \cite{complex}. These concepts have subsequently been employed---for instance, in conjunction with Efimov's seminal work on the K-theory of dualizable \cates \cite{efimov2024k}---to study the \emph{algebraic K-theory} of rigid analytic varieties \cite{andreychev2023ktheorieadischerraume}. We utilize nuclear \cates in this paper because they are rigid in our setting (see \Cref{prop: Nuc(X) is rigid}). This rigidity, in particular, provides a simple description of $\Nuc(k)$-atomic objects in $\Nuc(X)$ in terms of perfect complexes over $X$ (see \Cref{lem: perfect complexes over rigid varieties}). As explained in the following section, this is the key property that allows us to relate the smoothness of $X$ to the categorical notion of smoothness for $\Nuc(X)$ (see \Cref{thm: main theorem smoothness}). Combined with \Cref{thm: intro algebraization via perf}, this property will be crucial in relating the algebraization of a rigid analytic variety $X$ to the atomic generation of $\Nuc(X)$.

To obtain this result, another fundamental step consists of relating the geometric definition of smoothness for a rigid analytic variety $X$ to a categorical definition. The idea is to work relative to a base $\infty$-category $\mathscr{C}$ (in this case, $\Nuc(k)$), where the tensor product in $\LPr_{\mathscr{C}}$ satisfies a K\"unneth formula for rigid varieties. In our setting, this property is proven in \Cref{prop: kunneth for nuclear moduels}. This motivates us to consider presentable \cates that are linear with respect to the action of a base symmetric monoidal, stable, presentable \cat $\mathscr{V}$. In \Cref{section 1}, we review the necessary theory to work with these objects, primarily following \cite{BenMoshe2024, ramzi2024dualizablepresentableinftycategories,hinich2021yonedalemmaenrichedinfinity, heine2023, Gepner2015}, which provide a rigorous treatment of the subject. Furthermore, we generalize the notion of smooth \cates in this setting via the following definition, which is also considered in \cite{ramzi2024dualizablepresentableinftycategories}.

\begin{definition}
    Let $\cC$ be a presentable and dualizable \cat over $\cV$. We say that $\cC$ is \emph{internally smooth over $\cV$} if the coevaluation map admits a $\mathscr{V}$-linear right adjoint that preserves colimits.
\end{definition}

Within the analytic geometry framework described in this paper, we can now relate the categorical notion of smoothness to its geometric counterpart.

\begin{theorem}[\Cref{thm: main theorem smoothness}]\label{thm: main theomre smoothness intro}
    Let $f \colon X \to \Sp(k)$ be a (classical) rigid analytic variety over $k$. Then $X$ is smooth if and only if the \cat $\Nuc(X)$ is internally smooth in $\LPr_{\Nuc(k)}$.
\end{theorem}

This theorem follows as a consequence of \Cref{cor: non commutative smoothness global case}, which demonstrates that $X$ is smooth if and only if $\Delta_\ast \mathscr{O}_X$ is a perfect complex on $X\widehat{\times}X$, where $\Delta \colon X \to X\widehat{\times}X$ is the diagonal map. These results are obtained by combining the K\"unneth formula (\Cref{prop: kunneth for nuclear moduels}) and the results for smooth \cat in a six-functor formalism developed in \Cref{section 1}. In our setting, the construction of an appropriate six-functor formalism for nuclear \cates is described in \Cref{thm: 6 functor for nuclear modules in rigid geometry}. Specifically, this six-functor formalism is constructed using the theory developed in \cite{mann2022padic6functorformalismrigidanalytic} and revisited in \cite{heyer20246functorformalismssmoothrepresentations}. Furthermore, it is inspired by the formalism described in \cite[Lecture 8]{scholze2025sixfunctorformalisms} within the setting of derived algebraic geometry. We observe that, although this theorem ultimately concerns the study of a classical property of rigid analytic varieties, obtaining the categorical (or non-commutative) characterization of smoothness requires entering the realm of condensed mathematics and the setting of derived rigid analytic geometry described herein. Indeed, these methods are also essential for obtaining the requisite six-functor formalism.

As previously explained, this theorem plays a crucial role in relating the atomic generation of $\Nuc(X)$ to the algebraization of the rigid analytic variety $X$.

\subsection{Applications: atomic generation and algebraization}
Together with the notion of smooth and proper \cates, which are closely related to their geometric counterparts, one of the main features of non-commutative algebraic geometry is the existence of a single compact generator for the \cat of quasi-coherent sheaves on a quasi-compact and quasi-separated scheme $X$ over a field $k$. In particular, there exists a non-commutative $k$-algebra $A$ such that we have an equivalence of \cates 
\[
\Qcoh(X) \simeq \LMod_{A}.
\]

These results have been established, for example, in \cite{bondal2002generatorsrepresentabilityfunctorscommutative} and, within the setting of derived algebraic geometry, in \cite{toen2011derivedazumayaalgebrasgenerators}. 
Moreover, assuming that the variety $X$ is smooth and proper, one can deduce that the algebra $A$ is also smooth and proper. We will demonstrate how to reprove a similar result in \Cref{Appendix 1} using the theory of six-functor formalisms.   

More generally, the existence of a compact generator was proved by To\"en in the setting of dg-categories; the analogous result in the setting of \cates is the following.

\begin{theorem}[{\cite[Proposition 11.3.2.4]{SAG}}]\label{thm: compact generator for cpt generated cates}
    Let $R$ be an $\mathbb{E}_\infty$-ring specturm and $\mathscr{C}$ be a stable, $R$-linear, smooth and compactly generated $\infty$-category. Then $\mathscr{C}$ has a single compact generator. 
\end{theorem}

Without the assumption of compact generation, whether the previous result holds remains an open conjecture, which can be stated as follows: 

\begin{conjecture}[{\cite[Conjecture 0.6]{ramzi2024dualizablepresentableinftycategories}}]\label{conj: congettura Einfty}
    Let $R$ be an $\mathbb{E}_\infty$-ring specturm and $\mathscr{C}$ be a smooth, stable, presentable $R$-linear $\infty$-category. Then $\mathscr{C}$ admits a single compact generator.
\end{conjecture}

This conjecture has been verified in some particular cases in \cite{stefanich2023classification}. However, in the generality presented above, the conjecture remains open to the author's knowledge. 

In \cite{ramzi2024dualizablepresentableinftycategories, ramzi2024locallyrigidinftycategories}, Maxime Ramzi developed the theory of atomic objects and dualizable \cates; see also \cite{efimov2024k, efimov2025localizing} for closely related results and their applications to the theory of localizing invariants. In this context, atomic objects can be viewed as a variant of compact objects suitably adapted to the setting of \emph{enriched \cates}.

Building upon the notion of a $\mathscr{V}$-atomically generated \cat, as detailed in \Cref{def: categorie relativamente generate}, Maxime Ramzi proposed a generalization of the preceding conjecture to the setting of $\mathscr{V}$-enriched \cates.

\begin{conjecture}\label{conj: congettura setting enriched}
    Let $\mathscr{V}$ be a presentable symmetric monoidal \cat and let $\mathscr{C}$ be a $\mathscr{V}$-linear presentable \cat, internally smooth over $\mathscr{V}$. Then $\mathscr{C}$ is atomically generated.
\end{conjecture}

In this paper, we obtain a counterexample to this conjecture by studying the nuclear $\infty$-category of a smooth and proper rigid analytic variety. Indeed, we prove the following result.

\begin{theorem}[{\Cref{thm: algebraization via generator}}]\label{thm: main theorem intro}
Let $p\colon X \to \Sp(k)$ be a smooth and proper, quasi-compact, and separated rigid analytic variety over $\Spk$. Then $X$ is algebraizable if and only if $\Nuc(X)$ admits a single $\Nuc(k)$-atomic (equivalently, compact) generator.
\end{theorem}

If we assume $X$ to be algebraizable (i.e., $X \simeq Y^{\mathrm{an}}$ for some scheme $Y$), we can prove that $\Nuc(X)$ is $\Nuc(k)$-atomically generated. Building on the relative GAGA theorem of \cite{wang2026relativegagatheoremapplication}, we establish the following equivalence in \Cref{cor: GAGA for nuclear categories}: 
\[
\Nuc(Y^{\mathrm{an}}) \simeq \Qcoh(Y) \otimes_{\Mod_k}\Nuc(k).
\]
We then deduce our statement regarding the compact generation of $\Qcoh(Y)$. Conversely, if we assume that $\Nuc(X)$ admits a single $\Nuc(k)$-atomic generator, we can use \Cref{lem: perfect complexes over rigid varieties} to deduce that this generator must be a perfect complex. Furthermore, applying \Cref{thm: main theomre smoothness intro}, we infer that $\Perf(X)$ must also be smooth; in this way, we reduce our algebraization result to \Cref{thm: intro algebraization via perf}. 

To obtain the announced counterexample, we will also utilize the following result, which can be viewed as a generalization of \Cref{thm: compact generator for cpt generated cates} to the enriched setting.

\begin{proposition}[{\Cref{thm: existence of a single compact generator}}]
    Let $\cV$ be a commutative algebra in $\LPr_\mathrm{st}$ with compact unit, and let $\cC$ be a smooth \cat over $\cV$ which is $\cV$-atomically generated. Then $\cC$ admits a single $\cV$-atomic (equivalently, compact) generator.
\end{proposition}

With these results, we construct a counterexample to \Cref{conj: congettura setting enriched} and demonstrate that it is not reasonable to extend \Cref{conj: congettura Einfty} over a generic rigid base different from $\Mod_R$. Indeed, if we consider a non-algebraizable smooth and proper rigid analytic variety $X$, see for example the p-adic Hopf Surface studied in \cite{mustafin1977padic, voskuil1991nonarchimedean}, combining \Cref{thm: main theomre smoothness intro} and \Cref{thm: main theorem intro} reveals that $\Nuc(X)$ cannot admit a single $\Nuc(k)$-atomic generator. Thus, by \Cref{thm: existence of a single compact generator}, it cannot be $\Nuc(k)$-atomically generated. 

We have been informed that recently Maxime Ramzi has also constructed a counterexample to \Cref{conj: congettura setting enriched} in the case where $\mathscr{V}$ is an \cat of the form $\mathrm{Shv}(X,\mathrm{Sp})$. 

We conclude the intruduction observing that \Cref{thm: main theomre smoothness intro} implies that it is not, in general, possible to glue local, relatively compact generators of a sheaf of categories in the analytic (or fpqc) topology to obtain a global compact generator. Locally in the analytic topology over $X$, the sheaf $\Nuc(-)$ is of the form $\Mod_{A}(\Nuc(k))$ for affinoid algebras $A$. However, \Cref{thm: algebraization via generator} demonstrates that, globally, such a generator algebra does not necessarily exist.
    
This property of gluing local compact generators is one of the foundational features of \emph{Non-Commutative Algebraic Geometry}, and it highlights a key distinction between algebraic and analytic topologies. This gluing procedure for compact generators has been shown to hold for various topologies utilized in algebraic geometry; see, for example, \cite{toen2007moduliobjectsdgcategories} and \cite{bondal2002generatorsrepresentabilityfunctorscommutative}.

\subsection{Linear overview}
We begin \Cref{section 1} with some remarks on presentable categories enriched over a symmetric monoidal base category $\mathscr{V}$. In this section, we also describe the notion of an internally smooth (respectively, proper) category and compare this definition with the notion of a smooth (respectively, proper) algebra in $\mathscr{V}$. We conclude the section by analyzing the concept of a smooth (respectively, proper) category in the presence of a six-functor formalism.

In \Cref{Section 2} and \Cref{Section 3}, we introduce the definitions of animated affinoid algebras and derived rigid analytic spaces. We construct a six-functor formalism for nuclear categories over these spaces and describe the main properties of both the spaces and their associated nuclear categories.

In \Cref{section 4}, we compare the categorical notion of smoothness obtained in this setting with the classical notion of geometric smoothness for rigid varieties. Finally, we conclude by relating the algebraizability of a rigid variety $X$ to the existence of a $\Nuc(k)$-atomic generator in the category $\Nuc(X)$.

In \Cref{Appendix 1}, utilizing the results established in \Cref{section 1}, we reprove known results concerning smooth and proper categories in (derived) algebraic geometry.

In \Cref{Appendix 2}, we demonstrate how the framework developed in \Cref{section 1} can be applied to yield examples of (internally) smooth and proper \cates within the context of analytic stacks.

\subsection{Acknowledgements}
I would like to express my gratitude to my supervisors, Mauro Porta and Nicolò Sibilla, for their constant support, guidance, and invaluable suggestions. I am also very grateful to Maxime Ramzi, Juan Esteban Rodriguez Camargo, Qixiang Wang, Maximilian Hauck, Julius Mann, Emanuele Pavia, and Enrico Lampetti for many useful conversations. Part of this work was carried out during a visit to the Max Planck Institute for Mathematics in Bonn; I would like to thank the Institute for its warm hospitality and for providing excellent working conditions.

\subsection{Notation}
We generally adopt the notation and conventions introduced by Lurie in \cite{HTT,HA}. In particular, we denote by $\Sp$ the $\infty$-category of spectra, and by $\Mod_{R}$ (resp. $\LMod_{R}$) the $\infty$-category of modules over an $\bbE_{\infty}$-algebra (resp. left modules over an $\bbE_{1}$-algebra) in $\Sp$. We depart from these conventions only in using $\Ani$ to denote the $\infty$-category of spaces (anima).

Throughout this work, we fix a presentably symmetric monoidal $\infty$-category $\cV$. We let $\LPrV$ denote the $(\infty,2)$-category $\Mod_{\cV}(\LPr_{\mathrm{st}})$ of $\cV$-tensored, presentable, stable $\infty$-categories and $\cV$-linear colimit-preserving functors. By abuse of terminology, we refer to an object of $\LPrV$ simply as a \emph{$\cV$-linear category}.

If $\cC$ is a $\cV$-linear symmetric monoidal $\infty$-category, we denote its internal hom functor by $\uline{\Hom}_{\cC}(-,-)$, while $\mathrm{hom}^{\cV}(-,-)$ denotes the $\cV$-enriched hom functor as defined in \cite[Definition 4.20]{BenMoshe2024}. See \Cref{section 1} where we make precise these notations.

For any $\infty$-category $\cC$, we denote the usual mapping space (anima) by $\mathrm{Map}_{\cC}(-,-)$. If $\cC$ is stable, we denote the mapping spectrum by $\map_{\cC}(-,-)$.

Regarding the theory of condensed mathematics and analytic stacks, we follow the formalism developed in \cite{Clausen_Scholze_lectures}.

\section{Smooth and proper categories in a six-functor formalism}\label{section 1}
The notion of smooth and proper triangulated categories was introduced in \cite{kontsevich2008notes}.
Furthermore, in the setting of dg-categories (respectively, stable \cates), an analogous notion of smooth and proper categories has been studied, for example, in \cite{toen2007moduliobjectsdgcategories, efimov2018homotopyfinitenessdgcategories} (respectively, in \cite{SAG,Antieau_2014}).

In this section, we work over a stable, presentable, symmetric monoidal \cat $\cV$. We will employ a modification of the standard notion of smooth and proper \cates, adapted to the setting of $\Mod_{\cV}(\LPr)$.
We will compare these definitions with the notion of smooth and proper algebras defined in \cite[Definitions 4.6.4.2 and 4.6.4.13]{HA}.
Finally, at the end of the section, we explore how to obtain smooth \cates in the context of a six-functor formalism, with the aim of relating these results to the usual notion of smoothness in analytic geometry in subsequent sections.

\subsection{Remarks on presentable tensored categories}
In this section, we consider a stable, presentable, symmetric monoidal \cat $\cV$, and we denote its tensor product by $\otimes$. We let $\LPrV$ denote the $(\infty,2)$-category $\Mod_{\cV}(\LPr_{\mathrm{st}})$ and refer to an object in $\LPrV$ as a $\cV$-linear \cat. We now review some known facts and definitions in the setting of $\cV$-linear \cates. The main references for this topic are \cite{BenMoshe2024, ramzi2024dualizablepresentableinftycategories,hinich2021yonedalemmaenrichedinfinity, heine2023, Gepner2015}, which provide a rigorous treatment of the subject.

We begin by observing that the $(\infty,2)$-category $\LPrV$ admits a symmetric monoidal structure \cite[Theorem 4.5.2.1]{HA}; we denote its tensor product (induced by that of $\cV$) by $\otimesV$. Furthermore, any object $\cC$ in $\LPr_{\cV}$ is equipped with a \emph{$\cV$-enriched internal hom functor} $\mathrm{hom}_{\cC}^{\cV}(-,-)$, as defined in \cite[Definition 4.20]{BenMoshe2024}. We note that, for an object $x \in \cC$, the functor $\mathrm{hom}_{\cC}^{\cV}(x,-)$ is, in general, only lax $\cV$-linear. Indeed, it is defined as the right adjoint to the $\cV$-linear functor
\[
-\otimes_{\cV} x \colon \cV \to \cC,
\]
given by the action of $\cV$ on $\cC$. It follows from \cite[Proposition 4.17]{BenMoshe2024} that the right adjoint of a $\cV$-linear left adjoint functor is lax $\cV$-linear. When the \cat under consideration is clear from the context, we will simplify the notation and denote the $\cV$-enriched internal hom by $\mathrm{hom}^{\cV}(-,-)$.

\begin{proposition}
    Let $\cC$ and $\cD$ be two presentable $\cV$-tensored \cates, and let $L \colon \cC \to \cD$ be a $\cV$-linear left adjoint. Then $R \colon \cD \to \cC$ is the right adjoint of $L$ if and only if there is a natural isomorphism
    \begin{equation*}
        \mathrm{hom}^{\cV}(L(-),-) \xrightarrow{\sim} \mathrm{hom}^{\cV}(-,R(-)).
    \end{equation*}
\end{proposition}
\begin{proof}
    One direction is proved in \cite[Proposition 4.21]{BenMoshe2024}. Conversely, suppose that the canonical map
    \begin{equation*}
        \mathrm{hom}^{\cV}(L(-),-) \to \mathrm{hom}^{\cV}(-,R(-))
    \end{equation*}
    is a natural isomorphism. Applying the mapping space functor $\mathrm{Map}_{\cV}(\boldone_{\cV}, -)$ yields a natural isomorphism of spaces
    \begin{equation*}
        \mathrm{Map}(L(-), -) \simeq \mathrm{Map}(-,R(-))
    \end{equation*}
    (see \cite[Remark 7.2.12]{Gepner2015}).
\end{proof}

\begin{definition}[{\cite[Proposition 5.2]{BenMoshe2024}}]\label{def: internally left adjoint functors}
    We say that a $\cV$-linear left adjoint functor $L\colon \cC \to \cD$ is \emph{internally left adjoint} if its lax $\cV$-linear right adjoint $R\colon \cD \to \cC$ is $\cV$-linear and is itself a left adjoint.
\end{definition}

\begin{definition}\label{def: relatively compact}
    Let $\cC$ be a \cat in $\LPr_{\cV}$ and let $x$ be an object of $\cC$. We say that $x$ is \emph{compact relative to $\cV$} if the functor $\mathrm{hom}_{\cC}^{\cV}(x,-)$ preserves colimits.
\end{definition}

Note that since we always work with stable \cates, the previous definition is equivalent to requiring that $\mathrm{hom}_{\cC}^{\cV}(x,-)$ commutes with filtered colimits.

\begin{definition}\label{def: atomic object}
    Let $\cC$ be an $\bbE_{\infty}$-algebra in $\Mod_{\cV}(\LPrst)$ and let $x$ be an object of $\cC$. We say that $x$ is \emph{$\cV$-atomic} if it is compact relative to $\cV$ and the functor $\mathrm{hom}^{\cV}(x,-)$ is $\cV$-linear. Equivalently, the functor $-\otimes_{\cV} x$ is internally left adjoint.
\end{definition}

\begin{lemma}\label{lem: atomic vs relatively compact}
    Let $\cV$ be generated under colimits by dualizable objects and let $\cC$ be a symmetric monoidal presentable $\cV$-tensored \cat. Then an object $x \in \cC$ is compact relative to $\cV$ if and only if it is $\cV$-atomic.
\end{lemma}
\begin{proof}
    Every $\cV$-atomic object is compact relative to $\cV$ by definition. Conversely, assume $x$ is compact relative to $\cV$. By \cite[Corollary 3.8]{ramzi2024locallyrigidinftycategories}, for any $x, y \in \cC$ and any dualizable object $v \in \cV$, the canonical map
    \[
    v\otimes\mathrm{hom}^{\cV}(x,y) \to \mathrm{hom}^{\cV}(x,v\otimes y)
    \]
    is an equivalence. Since $\cV$ is generated under colimits by dualizable objects, any object $w \in \cV$ can be written as a colimit of dualizable objects, say $w \simeq \colim_{i \in I} v_{i}$. Using the assumption that $x$ is compact relative to $\cV$ (i.e., that $\mathrm{hom}^{\cV}(x,-)$ preserves colimits), we obtain the following equivalences:
    \begin{align*}
        \mathrm{hom}^{\cV}(x,w \otimes y) & \simeq \mathrm{hom}^{\cV}(x, (\colim_{i \in I} v_{i}) \otimes y) \\
        & \simeq \mathrm{hom}^{\cV}(x, \colim_{i \in I} (v_{i} \otimes y)) \\
        & \simeq \colim_{i \in I} \mathrm{hom}^{\cV}(x, v_{i} \otimes y) \\
        & \simeq \colim_{i \in I} (v_{i} \otimes \mathrm{hom}^{\cV}(x,y)) \\
        & \simeq (\colim_{i \in I} v_{i}) \otimes \mathrm{hom}^{\cV}(x,y) \\
        & \simeq w \otimes \mathrm{hom}^{\cV}(x,y).
    \end{align*}
    This shows that $\mathrm{hom}^{\cV}(x,-)$ is $\cV$-linear, concluding the proof.
\end{proof}

\begin{remark}
    We observe that the hypotheses of \Cref{lem: atomic vs relatively compact} are satisfied when $\cV$ is equivalent to $\Mod_{A}$ for an $\bbE_{\infty}$-algebra $A$ in $\mathrm{Sp}$ or when $\cV$ is equivalent to $\Nuc(A,A^{+})$ for a Huber pair $(A,A^{+})$ (see \cite[Satz 3.16]{andreychev2023ktheorieadischerraume} and \cite[Proposition 5.35]{andreychev2021pseudocoherent}).
\end{remark}

Moreover, when $\cV$ is the \cat of spectra (denoted by $\mathrm{Sp}$) or the \cat $\Mod_{R}$ of modules over an $\bbE_{\infty}$-ring $R$, we have the following characterization of $\cV$-atomic objects.

\begin{proposition}\label{prop: compact over Sp}
    Let $\cV$ be the \cat of spectra or $\Mod_{R}$ for an $\bbE_{\infty}$-ring $R$. Let $\cC$ be a \cat in $\LPr_{\cV}$ and let $x$ be an object of $\cC$. Then the following are equivalent:
    \begin{enumerate}
        \item The object $x$ is compact in $\cC$.
        \item The object $x$ is compact relative to $\cV$.
        \item The object $x$ is $\cV$-atomic.
    \end{enumerate}
\end{proposition}

\begin{proof}
    The equivalence $(2) \Leftrightarrow (3)$ follows from \Cref{lem: atomic vs relatively compact}. We now prove the equivalence $(1) \Leftrightarrow (2)$ in the case where $\cV$ is the \cat of spectra. An analogous argument is used in \cite[Proposition 2.8]{Ben_Moshe_2023}.
    
    Assume that $x$ is compact relative to $\mathrm{Sp}$. The mapping space $\Map_{\cC}(x,-)$ can be written as the delooping of the mapping spectrum $\Omega^{\infty}\mathrm{map}_{\cC}(x,-)$. Since the functor $\Omega^{\infty}$ preserves filtered colimits, this implies that $x$ is compact.
    
    Conversely, suppose that $x$ is compact in $\cC$. Observe that the shift functor $\Sigma^{n}\colon \mathrm{Sp} \to \mathrm{Sp}$ preserves colimits for every integer $n$. Consequently, the functor
    \[
    \Omega^{\infty}\Sigma^{n}\mathrm{map}_{\cC}(x,-) \simeq \Map_{\cC}(x,\Sigma^{n}-)
    \]
    preserves colimits. The statement now follows from the fact that the functors
    \[
    \Omega^{\infty}\Sigma^{n}: \mathrm{Sp} \to \Ani
    \]
    are jointly conservative.
    
    The equivalence $(1) \Leftrightarrow (2)$ when $\cV = \Mod_{R}$ follows from the case where $\cV = \mathrm{Sp}$, using the fact that for every objects $c$ and $c'$ in $\cC$ we have a functorial equivalence 
    \[
    \map_{\cC}(c,c') \simeq \map_{\Mod_{R}}(R, \hom^{\Mod_{R}}(c,c'))
    \]
    and $R$ is a compact generator of $\Mod_{R}$.
\end{proof}

Henceforth, for a stable $\infty$-category, implicitly using \Cref{prop: compact over Sp}, we will no longer distinguish between compact, $\mathrm{Sp}$-compact, and $\mathrm{Sp}$-atomic objects.

\begin{definition}\label{def: categorie relativamente generate}
    Let $\cC$ be a \cat in $\LPrV$, and $S$ a subset of objects in $\cC$. We denote by $\cC_{S}$ the full subcategory of $\cC$ spanned by the objects in $S$.
    \begin{enumerate}
        \item We say that $\cC$ is \emph{$\cV$-relatively generated by $S$} if $\cC$ coincides with the smallest full subcategory of itself that contains $\cC_{S}$ and is closed under colimits and the action of $\cV$.
        \item We say that $\cC$ is \emph{$\cV$-atomically generated} if it is $\cV$-relatively generated by its subset of $\cV$-atomic objects.
        \item We say that $\cC$ is \emph{$\cV$-relatively compactly generated} if it is $\cV$-relatively generated by its subset of $\cV$-compact objects.
    \end{enumerate}
\end{definition}

We note that our definition of a $\cV$-atomically generated \cat coincides with the notion of a $\cV$-molecular \cat introduced in \cite[Definition 2.10]{BenMoshe2024}.

\begin{remark}\label{rmk: V-enriched presheaves}
   We also observe that the set of $\cV$-atomic objects is small \cite[Proposition 5.9]{BenMoshe2024}. This allows us to ``control'' a $\cV$-atomically generated \cat by a small subcategory. Indeed, by applying \cite[Observation 1.28]{ramzi2024dualizablepresentableinftycategories}, it follows that if $\cM \in \LPrV$ is $\cV$-atomically generated, there is an equivalence 
\[
\cM \simeq \cP_\cV((\cM)^\mathrm{at}),
\]
where $\cP_{\cV}(-)$ denotes the category of $\cV$-enriched presheaves as defined in \cite{heine2023}, and $(\cM)^{\mathrm{at}}$ denotes the full subcategory of atomic objects in $\cM$. Moreover, it follows from \cite[Section 6]{BenMoshe2024} that $\cP_{\cV}(-)$ commutes with colimits.
\end{remark}

\begin{remark}\label{rmk: relative generators and generated under colimits and action}
    It follows from \cite[Corollary 3.13]{reutter2025enrichedinftycategoriesmarkedmodule} that $\cC$ is \emph{$\cV$-relatively generated by $S$} as in \Cref{def: categorie relativamente generate} if and only if the family of functors
    \[
    \{\hom_{\cC}^{\cV}(s,-)\colon \cC \to \cV\}_{s \in S}
    \]
    is jointly conservative.
\end{remark}

We will now consider the case where $\cV$ is a rigid \cat, in which the various definitions introduced above coincide.

In this paper, we extensively use the notion of rigid \cates as introduced in \cite{gaitsgory2019study}. The main feature is that when $\cV$ is rigid, $\LPrV$ enjoys several favorable properties (see, for example, \cite[Chapter 1, Sec 9]{gaitsgory2019study} and \cite[Lemma 3.33]{anschutz2024descentsolidquasicoherentsheaves}). Another motivation for introducing rigid \cates comes from rigid analytic geometry; indeed, the \cat of nuclear sheaves in this setting is rigid. This \cat will be studied in detail later in this paper. 

\begin{proposition}\label{prop: compact objects over rigid categories}
    Let $\cV$ be a stable, presentable, symmetric monoidal, rigid \cat, let $\cC$ be a \cat in $\LPrV$, and let $x$ be an object in $\cC$. Then the following are equivalent:
    \begin{enumerate}
        \item The object $x$ in $\cC$ is compact.
        \item The object $x$ in $\cC$ is compact relative to $\Sp$.
        \item The object $x$ in $\cC$ is compact relative to $\cV$.
        \item The object $x$ in $\cC$ is $\cV$-atomic.
    \end{enumerate}
\end{proposition}
\begin{proof}
    This follows by combining \Cref{prop: compact over Sp} and \cite[Corollary 9.3.4, Lemma 9.3.6]{gaitsgory2019study}.
\end{proof}

From the above proposition, we can immediately deduce the following.

\begin{corollary}
    If $\cV$ is rigid, then the definitions of a $\cV$-atomically generated \cat and a $\cV$-compactly generated \cat coincide.
\end{corollary}

For this reason, we will not further distinguish between the two notions when working over a rigid base.

\subsection{Reminder on smooth and proper categories and algebras}
This subsection introduces the notions of smooth and proper enriched \cates that will be utilized throughout this paper. We also prove the main result of this section, \Cref{thm: existence of a single compact generator}, which plays a central role in refuting \Cref{conj: congettura setting enriched}.

\begin{definition}\label{def: smooth algebras}
    Let $A$ be an $\bbE_{1}$-algebra in $\cV$. We say that $A$ is \emph{smooth} if the diagonal module is dualizable as a module over $A \otimes_{\cV} A^{\mathrm{op}}$.
\end{definition}

\begin{definition}[{\cite[Definition 4.6.4.2]{HA}}]\label{def: proper algebras}
    Let $A$ be an $\bbE_{1}$-algebra in $\cV$. We say that $A$ is \emph{proper} if it is dualizable as an object of $\cV$.
\end{definition}

\begin{definition}\label{def: dualizable categories}
    Let $\cC$ be a $\cV$-linear \cat. We say that $\cC$ is \emph{dualizable} over $\cV$ if there exists a \cat $\cC^{\vee}$ in $\LPrV$, together with two functors in $\LPrV$, called the evaluation and coevaluation functors:
    \begin{equation*}
        \mathrm{ev}: \cC \otimesV \cC^{\vee} \to \cV
    \end{equation*}
    \begin{equation*}
        \mathrm{coev}: \cV \to \cC^{\vee} \otimesV \cC,
    \end{equation*}
    such that the composite maps
    \begin{equation}\label{eq: compatibility 1}
        \begin{tikzcd}
            \cC && {\cC\otimesV \cC^{\vee} \otimesV \cC} && \cC
            \arrow["{\mathrm{id}\boxtimes \mathrm{coev}}", from=1-1, to=1-3]
            \arrow["{\mathrm{ev} \boxtimes \mathrm{id}}", from=1-3, to=1-5]
        \end{tikzcd}    
    \end{equation}
    \begin{equation}\label{eq: compatibility 2}
        \begin{tikzcd}
            {\cC^{\vee}} && {\cC^{\vee} \otimesV \cC \otimesV \cC^{\vee}} && {\cC^{\vee}}
            \arrow["{\mathrm{coev} \boxtimes \mathrm{id}}", from=1-1, to=1-3]
            \arrow["{\mathrm{id} \boxtimes \mathrm{ev}}", from=1-3, to=1-5]
        \end{tikzcd}
    \end{equation}
    are equivalent to the respective identity functors.
\end{definition}

\begin{remark}
    It is a general fact that in a closed symmetric monoidal \cat, the dual of a dualizable object $X$ can be identified with the internal hom $\uline{\Hom}_{\cC}(X,\boldone)$. In particular, in our setting, we can identify $\cC^{\vee}$ with $\Fun^{\mathrm{L}}_{\cV}(\cC,\cV)$ (see \cite[4.3.3]{gaitsgory2019study}).
\end{remark}

\begin{remark}
    This definition is well-posed since $\cV$ is a symmetric monoidal $\infty$-category. If we only assume $\cV$ to be an $\bbE_{1}$-algebra object in $\LPrV$, then one can employ the definition of a dualizable \cat as explained in \cite[Definition 1.10]{efimov2025localizing}. We will not discuss the non-symmetric monoidal case further, as we assume throughout that our base $\cV$ is symmetric monoidal.
\end{remark}

\begin{remark}
    We note that \Cref{eq: compatibility 1} implies that $\coev(\boldone_{\cV})$ classifies the identity functor under the equivalence 
    \[
    \cC^\vee \otimes_{\cV} \cC \simeq \Fun^{\mathrm{L}}_{\cV}(\cC, \cC).
    \]
\end{remark}

\begin{definition}\label{def: smooth categories}
    Let $\cC$ be a presentable and dualizable \cat over $\cV$. We say that $\cC$ is \emph{smooth over $\cV$} if the coevaluation map admits a right adjoint that preserves colimits.
\end{definition}

\begin{definition}\label{def: proper categories}
    Let $\cC$ be a presentable and dualizable \cat over $\cV$. We say that $\cC$ is \emph{proper over $\cV$} if the evaluation map admits a right adjoint that preserves colimits.
\end{definition}

Note that, in general, the right adjoints of the evaluation and coevaluation maps are only lax $\cV$-linear. This limitation motivates the introduction of the following variants of the notions of smooth and proper $\cV$-linear \cates.

\begin{definition}\label{def: internally smooth}
    Let $\cC$ be a presentable and dualizable \cat over $\cV$. We say that $\cC$ is \emph{internally smooth} over $\cV$ if the coevaluation map admits a $\cV$-linear right adjoint that preserves colimits.
\end{definition}

\begin{definition}\label{def: internally proper}
    Let $\cC$ be a presentable and dualizable \cat over $\cV$. We say that $\cC$ is \emph{internally proper over} $\cV$ if the evaluation map admits a $\cV$-linear right adjoint that preserves colimits.
\end{definition}

\begin{definition}\label{def: fully dualizable}
    Let $\cC$ be a dualizable \cat over $\cV$. We say that $\cC$ is \emph{fully dualizable} over $\cV$ if it is internally smooth and internally proper; that is, if both the evaluation and coevaluation maps admit a $\cV$-linear right adjoint that preserves colimits.
\end{definition}

\begin{remark}
    Observe that $\cC$ satisfies \Cref{def: smooth categories} and \Cref{def: proper categories} if and only if it is dualizable in the \cat $\mathrm{Pr}^{\mathrm{L,dual}}_{\cV}$, which consists of dualizable \cates over $\cV$ and $\cV$-linear strongly continuous functors. Indeed, this condition is equivalent to requiring that both the coevaluation and evaluation functors admit continuous right adjoints.
\end{remark}

\begin{remark}\label{rmk: smooth cat and compact unit}
    We observe that $\cC$ is smooth if and only if $\coev(\boldone_{\cV})$ is $\cV$-atomic, in particular if $\boldone_{\cV}$ is compact in $\cV$ then also $\coev(\boldone_{\cV})$ is compact. When $\cV$ is rigid, this is equivalent to requiring that $\coev(\boldone_{\cV})$ is compact (see \Cref{prop: compact objects over rigid categories}).
\end{remark}
 
Note that, \emph{a priori}, the right adjoint of a $\cV$-linear functor is not necessarily $\cV$-linear. However, this property holds in specific cases; for example, if $\cV$ is the \cat of spectra or, more generally, a \emph{mode} as explained in \cite[Example 1.13]{ramzi2024dualizablepresentableinftycategories}. In such cases, being fully dualizable over $\cV$ is strictly equivalent to being smooth and proper over $\cV$. We observe that this equivalence also holds when $\cV$ is rigid.

\begin{lemma}\label{lem: smooth and proper over a rigid base}
    Let $\cV$ be rigid and let $\cC$ be an object in $\LPrV$. Then $\cC$ is internally smooth (respectively, internally proper) if and only if it is smooth (respectively, proper).
\end{lemma}
\begin{proof}
    It suffices to verify that the right adjoint of a $\cV$-linear map is itself $\cV$-linear. This follows from the fact that $\cV$ is rigid (see \cite[Lemma 9.3.6]{gaitsgory2019study}).
\end{proof}

\begin{proposition}\label{thm: existence of a single compact generator}
     Let $\cV$ be a commutative algebra in $\LPr_\mathrm{st}$ with compact unit and let $\cC$ be a smooth \cat over $\cV$ which is $\cV$-atomically generated (equivalently, $\cV$-compactly generated). Then $\cC$ admits a single $\cV$-atomic (equivalently, compact) generator.
\end{proposition}

\begin{proof}
    Let $S$ denote the set of $\cV$-atomic objects of $\cC$. For every subset $T \subset S$, let $\cC_{T}$ be the full subcategory of $\cC$ spanned by $T$, and let $\widehat{\cC_{T}}$ be the smallest full subcategory of $\cC$ containing $\cC_{T}$ that is closed under colimits and the action of $\cV$. Since $S$ is the set of $\cV$-atomic generators, we have a canonical equivalence
    \[
    \widehat{\cC_{S}} \simeq \cC.
    \]
    We claim that the canonical map 
    \begin{equation}\label{eq: proof single atomic generator}
    \colim_{T \subset S} \widehat{\cC_{T}} \to \cC
    \end{equation}
    is an equivalence, where the colimit is taken over the filtered system of finite subsets of $S$. By \Cref{rmk: V-enriched presheaves}, each $\widehat{\cC_{T}}$ is $\cV$-atomically generated by $T$ and can be identified with the \cat of $\cV$-enriched presheaves $\cP_{\cV}(\cC_{T})$. Similarly, $\cC$ is identified with $\cP_{\cV}(\cC_{S})$. Since $\cP_{\cV}(-)$ is a left adjoint (see \Cref{rmk: V-enriched presheaves}), it preserves colimits; thus, to show that \eqref{eq: proof single atomic generator} is an equivalence, it suffices to observe that the canonical map
    \[
    \colim_{T \subset S} \cC_{T} \to \cC_{S}
    \]
    is a tautological equivalence.
    
    We now observe that the coevaluation functor
    \[
    \mathrm{coev} \colon \cV \to \cC^{\vee} \otimes_{\cV} \cC
    \]
    classifies the identity functor $\mathrm{id}_{\cC}$. Since $\cC$ is smooth and $\boldone_\cV$ is a compact object, $\mathrm{id}_{\cC}$ is a compact object in $\Fun^{\mathrm{L}}(\cC, \cC)$ (see \Cref{rmk: smooth cat and compact unit}). Using \eqref{eq: proof single atomic generator}, we can express the identity as a filtered colimit
    \[
    \mathrm{id}_{\cC} \simeq \colim_{T \subset S} f_{T},
    \]
    where $f_{T} \colon \widehat{\cC_{T}} \to \cC$ is the natural inclusion. The compactness of $\mathrm{id}_{\cC}$ implies that the identity natural transformation
    \[
    \boldone_{\mathrm{id}_{\cC}} \colon \mathrm{id}_{\cC} \to \mathrm{id}_{\cC}
    \]
    factors through $f_{T}$ for some finite subset $T \subset S$. This implies that $\mathrm{id}_{\cC}$ factors through $\widehat{\cC_{T}}$, and in particular, that $\cC$ coincides with $\widehat{\cC_{T}}$. It follows that $\cC$ is generated by the single compact object $\bigoplus_{a_{i} \in T} a_{i}$.
\end{proof}

We now turn to a detailed study of the case where the \cat is $\LMod_{A}(\cV)$, for $A$ an algebra object in $\cV$.

\subsection{Modules over an algebra in $\cV$}
In this subsection, we study the \cat $\LMod_{A}(\cV)$ in detail. In particular, we show that it is generated under colimits and the action of $\cV$. We also analyze its dualizable and compact objects. Finally, we compare the notion of internal smoothness for $\LMod_{A}(\cV)$ to the smoothness of $A$ as an algebra in $\cV$ (see \Cref{prop: smooth cates vs smooth algebras}). This subsection, along with its main result (\Cref{prop: smooth cates vs smooth algebras}), will be particularly relevant when we study the compact generation of the \cat of nuclear sheaves, specifically in the proof of \Cref{thm: algebraization via generator}.

\begin{definition}
    Let $R$ be an $\bbE_{1}$-algebra in $\cV$, and let $X$ be an object in $\LMod_{R}(\cV)$. An object $Y$ in $\RMod_{R}(\cV)$ is said to be \emph{left dual} to $X$ if there exists an adjunction
    \begin{equation*}
        \adjunction{X\otimes-}{\cV}{\LMod_{R}(\cV)}{Y\otimes_{R}-},
    \end{equation*}
    where the tensor product on the left is given by the action of $\cV$ on $\LMod_{R}(\cV)$, and the one on the right is described in \Cref{eq: evaluation fucntor module categories}.
\end{definition}

\begin{lemma}
    Let $R$ be an $\bbE_{1}$-algebra in $\cV$. Then $R$ is $\cV$-atomic in $\LMod_{R}(\cV)$. Furthermore, for every $\cV$-atomic object $v$ in $\cV$, the object $R\otimes v$ is $\cV$-atomic in $\LMod_{R}(\cV)$.
\end{lemma}
\begin{proof}
    The first assertion follows from the observation that $\mathrm{hom}^{\cV}_{\LMod_{R}}(R,-)$ can be identified with the forgetful functor
    \begin{equation*}
        \LMod_{R}(\cV) \to \cV.
    \end{equation*}
    This functor preserves colimits and is $\cV$-linear (see \cite[4.2.3.7]{HA}).
    Moreover, since the free module functor
    \begin{equation*}
        R\otimes-\colon \cV \to \LMod_{R}(\cV)
    \end{equation*}
    is an internal left adjoint, it preserves atomic objects.
\end{proof}

\begin{proposition}\label{prop: LModR is molecular}
    Let $R$ be an $\bbE_{1}$-algebra in $\cV$. Then $\LMod_{R}(\cV)$ is $\cV$-atomically generated. Moreover, if an \cat $\cC$ in $\LPrV$ is generated by a single $\cV$-atomic object $X$, then $\cC$ is equivalent to $\LMod_{R}(\cV)$, where $R$ is the $\bbE_{1}$-algebra $\hom^{\cV}(x,x)$.
\end{proposition}
\begin{proof}
    It suffices to show that $\LMod_{R}(\cV)$ is generated under colimits by objects of the form $R\otimes v$ with $v \in \cV$. This follows from \cite[4.7.3.14]{HA} by observing that $\LMod_{R}(\cV)$ is the category of modules over the monad $R\otimes-$. The latter statement follows from an application of the Barr--Beck--Lurie theorem.
\end{proof}

\begin{corollary}\label{cor: dualizable if and only if V-atomic}
    Let $R$ be an $\bbE_{1}$-algebra in $\cV$. Then an object $X$ in $\LMod_{R}(\cV)$ is dualizable if and only if it is $\cV$-atomic.
\end{corollary}
\begin{proof}
    Suppose that $X$ is dualizable in $\LMod_{R}(\cV)$. Then there exists an object $Y \in \RMod_{R}(\cV)$ and an equivalence
    \[
    \mathrm{hom}^{\cV}_{\LMod_{R}}(X,-) \simeq Y \otimes_{R}-.
    \]
    Since the functor $Y\otimes_{R}-$ preserves colimits and is $\cV$-linear (see \cite[4.4.3.14]{HA}), it follows that $X$ is $\cV$-atomic.
    Conversely, suppose that $X$ is $\cV$-atomic in $\LMod_{R}(\cV)$. We define
    \[
    X^{\vee} \coloneqq \mathrm{hom}^{\cV}_{\LMod_{R}}(X,R).
    \]
    This is canonically an object in $\RMod_{R}(\cV)$. To prove that $X^{\vee}$ is the dual of $X$, it suffices to show that the functor $X^{\vee}\otimes_{R}-$ is equivalent to $\mathrm{hom}^{\cV}_{\LMod_{R}}(X,-)$. This follows from \Cref{prop: LModR is molecular}: since both functors preserve colimits and are $\cV$-linear, it is enough to observe that they coincide when evaluated on $R$. Indeed, we have
    \[
    X^{\vee}\otimes_{R} R \simeq X^{\vee} = \mathrm{hom}^{\cV}_{\LMod_{R}}(X,R).
    \]
\end{proof}

We now turn to the notion of internal smoothness for the \cat $\LMod_{A}(\cV)$ and compare it to the smoothness of the algebra $A$. For instance, if $R$ is an algebra object in the \cat of spectra, then $\LMod_{R}$ is right dualizable with right dual $\LMod_{R^{op}}\simeq \RMod_{R}$. We will demonstrate that such results remain valid when working over our symmetric monoidal base $\cV$ instead of spectra. In particular, we prove that if $A$ is an $\bbE_{1}$-algebra in $\cV$, then the \cat of left $A$-modules, $\LMod_{A}(\cV)$, is always dualizable over $\cV$. Furthermore, if $\cV$ is rigid, then $\LMod_{A}(\cV)$ is also rigid.

We begin by observing that, as explained in \cite[4.1.1.7]{HA}, an involution can be defined on the $\infty$-operad \textbf{Assoc}, which allows for the definition of the opposite algebra for every algebra object in $\cV$. We identify $\RMod_{A}(\cV)$ with $\LMod_{A^{op}}(\cV)$ and show that this \cat is the dual of $\LMod_{A}(\cV)$.
To achieve this, we explicitly describe the coevaluation and evaluation functors and verify that they constitute a duality datum. 

The coevaluation functor is obtained as a $\cV$-linear map by sending the unit $\mathbbm{1}_{\cV}$ to $A$ equipped with the diagonal bimodule structure:
\begin{equation}\label{eq: covaluation fucntor module categories}
    \begin{tikzcd}[row sep=0.1in]
        {\mathrm{coev}\colon\cV} &&& {\LMod_{A}(\cV)\otimes \RMod_{A}(\cV)} \simeq \Mod_{A\otimesV A^{op}}(\cV) \\
        {\mathbbm{1}_{\cV}} &&& A_{\Delta}.
        \arrow[from=1-1, to=1-4]
        \arrow[maps to, from=2-1, to=2-4]
    \end{tikzcd}
\end{equation}
The diagonal bimodule structure on $A$ is obtained via the forgetful functor $m_{\ast} \colon \Mod_{A} \to \Mod_{A\otimes A^{op}}$ induced by the multiplication map $m \colon A\otimes A^{op} \to A$. It is straightforward to verify that the coevaluation functor has a right adjoint given by $\mathrm{hom}^{\cV}_{A\otimes A^{op}}(A_{\Delta},-)$.

The evaluation functor is defined via the relative tensor product; explicitly, it is given by the map:
\begin{equation}\label{eq: evaluation fucntor module categories}
    \begin{tikzcd}[row sep=0.1in]
        {\mathrm{ev} \colon \RMod_{A}(\cV)\otimes \LMod_{A}(\cV)} &&& \cV \\
        {M\boxtimes N} &&& {M \otimes_{A}N}. \\
        && {}
        \arrow[from=1-1, to=1-4]
        \arrow[maps to, from=2-1, to=2-4]
    \end{tikzcd}
\end{equation}
In particular, this functor sends $A^{op}\boxtimes A$ to $A$ (see \cite[Sec. 4.4]{HA} and \cite[4.2.1]{gaitsgory2019study}).

\begin{proposition}
    Let $\cV$ be a stable, presentable, symmetric monoidal \cat and let $A$ be an $\bbE_{1}$-algebra in $\cV$. Then $\LMod_{A}(\cV)$ is dualizable over $\cV$ with dual $\LMod_{A^{op}}(\cV)$.
\end{proposition}
\begin{proof}
    We need to check that the coevaluation and evaluation functors defined in \Cref{eq: covaluation fucntor module categories} and \Cref{eq: evaluation fucntor module categories} constitute a duality datum.
    We observe that the analogue of \Cref{eq: compatibility 1} is given by the composition
    \begin{equation*}
        \begin{tikzcd}
            {\LMod_{A}} && {\LMod_{A}\otimes \RMod_{A} \otimes \LMod_{A}} && {\LMod_{A}}
            \arrow["{\mathrm{coev} \otimes \mathrm{id}}", from=1-1, to=1-3]
            \arrow["{\mathrm{id} \otimes \mathrm{ev}}", from=1-3, to=1-5]
        \end{tikzcd}
    \end{equation*}
    which sends a module $M$ to $A_{\Delta}\otimes_{A} M$. By \cite[Proposition 4.6.2.17]{HA}, this composition is homotopic to the identity. Similarly, one can show that the analogue of the map in \Cref{eq: compatibility 2} is homotopic to the identity.
\end{proof}

\begin{proposition}\label{prop: smooth cates vs smooth algebras}
    Let $A$ be an $\bbE_{1}$-algebra in $\cV$. Then $A$ is smooth as an algebra if and only if the category of modules $\LMod_{A}(\cV)$ is internally smooth as a $\cV$-linear \cat. In particular, the notion of an internally smooth \cat is Morita invariant.
\end{proposition}
\begin{proof}
    By definition, $A$ is smooth as an algebra in $\cV$ if and only if $\mathrm{coev}(\boldone_{\cV}) \simeq A_{\Delta}$ is dualizable in $\LMod_{A\otimes_{\cV}A^{op}}(\cV)$. By invoking \Cref{cor: dualizable if and only if V-atomic}, we deduce that this holds if and only if $A_{\Delta}$ is $\cV$-atomic. In particular, this is equivalent to the condition that the right adjoint of the coevaluation map, given by
    \[
    \mathrm{hom}^{\cV}_{\LMod_{A\otimes A^{op}}}(A_{\Delta},-),
    \]
    preserves colimits and is $\cV$-linear.
\end{proof}

\subsection{Smooth categories in a six-functor formalism}

The notion of a \emph{six-functor formalism} has been extensively studied in recent years, particularly following the works of \cite{heyer20246functorformalismssmoothrepresentations,scholze2025sixfunctorformalisms, liu2024enhancedoperationsbasechange, zavyalov2023poincaredualityabstract6functor, kesting2025categoricalkunnethformulasanalytic}. These ideas have found significant applications in rigid analytic geometry and condensed mathematics, as illustrated, for instance, in \cite{mann2022padic6functorformalismrigidanalytic, camargo2024analyticrhamstackrigid, soor2025sixfunctorformalismquasicoherentsheaves}. Motivated by the objective of utilizing an appropriate $6$-functor formalism in analytic geometry to characterize smooth varieties, we first investigate the conditions under which smooth \cat arise within such a framework. This is established in \Cref{prop: smoothness in a 6 ff} and its subsequent corollaries. Specifically, in this section, we consider the following. 

\begin{setup}\label{assumptions 6 functor}
Let $\cC$ be an \cat arising from a geometric setup as defined in \cite[Definition 2.1.1]{heyer20246functorformalismssmoothrepresentations}, equipped with a class $E$ of $!$-able maps. Furthermore, let
\[
\cD \colon \cC \to \LPr_{\cD(X)}
\]
be a symmetric monoidal, stable, and presentable $6$-functor formalism, as in \cite[Definition 3.2.1]{heyer20246functorformalismssmoothrepresentations}.
\end{setup}

In this setup, employing the same argument as in \cite[Corollary 1.2.10]{montagnani2025axiomaticapproachanalytic1affineness}, we observe that for every $!$-able map $Y \to X$, the category $\cD(Y)$ is dualizable over $\cD(X)$. Indeed, utilizing the span diagrams
\[
\begin{tikzpicture}[scale=0.75,baseline=(current bounding box.center)]
\node (a) at (0,2){$Y$};
\node (b) at (-2,0){$X$};
\node (c) at (2,0){$Y\times_{X} Y$};
\draw[->,font=\scriptsize](a) to[bend right] node[left]{$f$}(b);
\draw[->,font=\scriptsize](a) to[bend left] node[right]{$\Delta$}(c);
\end{tikzpicture}\quad\text{ and }\quad\begin{tikzpicture}[scale=0.75,baseline=(current bounding box.center)]
\node (a) at (0,2){$Y$};
\node (c) at (2,0){$X$};
\node (b) at (-2,0){$Y\times_{X} Y$};
\draw[->,font=\scriptsize](a) to[bend left] node[right]{$f$}(c);
\draw[->,font=\scriptsize](a) to[bend right] node[left]{$\Delta$}(b);
\end{tikzpicture}
\]
we can explicitly express the evaluation and coevaluation maps as follows:
\begin{equation}\label{eq: coevaluation 6 functor formalism}
    \begin{tikzcd}
	{\coev\colon\cD(X)} && {\cD(Y)} && {\cD(Y\times_{X}Y)}
	\arrow["{f^*}"', from=1-1, to=1-3]
	\arrow["{\Delta_{!}}"', from=1-3, to=1-5]
\end{tikzcd}
\end{equation}
\begin{equation}\label{eq: evaluation 6 functor formalism}
    \begin{tikzcd}
	{\ev\colon\cD(Y\times_{X}Y)} && {\cD(Y)} && {\cD(X)}
	\arrow["{\Delta^*}"', from=1-1, to=1-3]
	\arrow["{f_!}"', from=1-3, to=1-5]
\end{tikzcd}
\end{equation}
Observe that the pullback functors are inherently linear. Moreover, the projection formula implies that $f_!$ is $\cD(X)$-linear, and \Cref{eq: Delta_! is linear} demonstrates that $\Delta_{!}$ is also $\cD(X)$-linear. Consequently, both the evaluation and coevaluation functors are $\cD(X)$-linear.

Their right adjoints admit explicit descriptions: the right adjoint to the evaluation functor is given by
\[
\ev^{\mathrm{R}}\simeq \Delta_{*} \circ f^!,
\]
while the right adjoint to the coevaluation functor is given by
\[
\mathrm{coev}^{\mathrm{R}} \simeq f_* \circ \Delta^!.
\]
To determine when the \cat $\cD(Y)$ is smooth over $\cD(X)$, we must investigate the conditions under which the functor $\mathrm{coev}^{\mathrm{R}}$ preserves colimits.

\begin{proposition}\label{prop: smoothness in a 6 ff}
    Let $\cC$ and $\cD$ be as in \Cref{assumptions 6 functor}, and let $f\colon Y \to X$ be a $!$-able map in $\cC$. Then the category $\cD(Y)$ is smooth over $\cD(X)$ (internally smooth) provided the following conditions are satisfied:
    \begin{enumerate}
        \item The functor $f_*$ preserves colimits (and, moreover, is $\cD(X)$-linear).
        \item The object $\Delta_{!}\boldone_{\cD(Y)}$ is relatively compact over $\cD(Y)$ in $\cD(Y \times_X Y)$ (and, moreover, is $\cD(Y)$-atomic).
        \item Letting $\mathrm{pr} \colon Y \times_{X} Y \to Y$ denote one of the projections, we require $\mathrm{pr}_{\ast}$ to preserve colimits (and, moreover, to be $\cD(X)$-linear).
    \end{enumerate}
\end{proposition}

\begin{proof}
We prove only the smoothness of $\cD(Y)$ over $\cD(X)$; internal smoothness follows by an analogous argument.
Using the projection formula, we observe that $\Delta_{!}(-)$ acts as a $\cD(Y)$-linear functor. Indeed, for all $M,N$ in $\cD(Y)$, we have the following natural equivalences:
\begin{equation}\label{eq: Delta_! is linear}
    \Delta_{!}(M\otimes N) \simeq \Delta_{!}(M \otimes \Delta^{\ast}\mathrm{pr}^{\ast}N) \simeq \Delta_{!}(M)\otimes \mathrm{pr}^{\ast}(N).
\end{equation}
In particular, for every $M$ in $\cD(Y)$, we can express $\Delta_{!}(M)$ as follows:
\[ 
\Delta_{!}(M) \simeq \Delta_{!}(\boldone_{\cD(Y)}\otimes M)\simeq \Delta_{!}\boldone_{\cD(Y)}\otimes \mathrm{pr}^{\ast}(M).
\]
It follows that its right adjoint can be explicitly written as 
\[
\mathrm{pr}_{\ast} \uline{\Hom}_{\cD(Y\times_{X} Y)}(\Delta_{!}\boldone_{\cD(Y)},(-)).
\]
Consequently, by invoking the hypotheses, this right adjoint preserves colimits, thus defining a functor in $\LPr$. By condition (1), $f_{\ast}$ also belongs to $\LPr$.
\end{proof}

We can specialize \Cref{prop: smoothness in a 6 ff} to two cases of particular interest. To do this, we first recall the definitions of suave and prim maps. 

\begin{definition}\cite[Lemma 4.4.6]{heyer20246functorformalismssmoothrepresentations}\label{def: prim maps}
Let $\cD$ be a $6$-functor formalism as in \Cref{assumptions 6 functor}, let $f\colon X \to S$ be a $!$-able map, and let $P \in \mathcal{D}(X)$. Then $P$ is said to be $f$-prim if and only if the natural map
\[
    f_! \left( \pi_{2 \ast} \uline{\Hom}(\pi_{1}^* P, \Delta_! \boldone) \otimes P \right) \longrightarrow f_* \uline{\Hom}(P, P)
\]
becomes an isomorphism after applying $\mathrm{Hom}(\boldone, -)$. Here, $\pi_{i}\colon X \times_S X \to X$ denote the two projections, and $\Delta\colon X \to X \times_S X$ is the diagonal map.
Moreover, if $P$ is $f$-prim, we denote by $\mathbb{D}_f(P)$ the object $\pi_{2 \ast} \uline{\Hom}(\pi_{1}^* P, \Delta^! \boldone)$.\\ We say that $f$ is $\cD$-prim if $\boldone_{X}$ is $f$-prim. In this case, we denote by $\delta_{f}$ the object 
$\mathbb{D}_f(\boldone_{X})$.
\end{definition}

\begin{definition}\cite[Lemma 4.4.5]{heyer20246functorformalismssmoothrepresentations}\label{def: suave maps}
Let $\cD$ be a $6$-functor formalism as in \Cref{assumptions 6 functor}, let $f\colon X \to S$ be a $!$-able map, and let $P \in \cD(X)$. Then $P$ is said to be $f$-suave if and only if the natural map
\[
    \pi_{1}^* \uline{\Hom}(P, f^! \boldone) \otimes \pi_{2}^* P \longrightarrow \uline{\Hom}(\pi_{1}^* P, \pi_{2}^! P)
\]
becomes an isomorphism after applying $\mathrm{Hom}(\boldone, \Delta^!(-))$. In this case, we define $\mathbb{SD}_f(P)$ to be $\uline{\Hom}(P, f^! \boldone)$. \\ We say that $f$ is $\cD$-suave if $\boldone_{X}$ is $f$-suave. In this case, we denote by $\omega_{f}$ the object $\mathbb{SD}_f(\boldone)$.
\end{definition}

\begin{proposition}\cite[Corollary 4.5.11]{heyer20246functorformalismssmoothrepresentations}\label{prop: prima and suave maps}
Let $\cD$ be a $6$-functor formalism as in \Cref{assumptions 6 functor}, and let $f\colon X \to Y$ be a $!$-able map.
\begin{enumerate}
    \item If $f$ is $\mathcal{D}$-suave, then the natural maps
    \[
        \omega_f \otimes f^*(-) \xrightarrow{\sim} f^!(-)  \quad \text{and} \quad f^*(-) \xrightarrow{\sim}  \uline{\Hom}(\omega_f, f^!(-))
    \]
    are isomorphisms of functors $\cD(Y) \to \cD(X)$.
     
    \item If $f$ is $\mathcal{D}$-prim, then the natural maps
    \[
        f_!(\delta_f \otimes -) \xrightarrow{\sim} f_*(-) \quad \text{and} \quad f_!(-) \xrightarrow{\sim} f_* \uline{\Hom}(\delta_f, -)
    \]
    are isomorphisms of functors $\cD(X) \to \cD(Y)$.
\end{enumerate}
\end{proposition}

\begin{corollary}\label{cor: smooth categories 6ff suave and prim case}
    Let $\cC$ and $\cD$ be as in \Cref{assumptions 6 functor}, and let $f \colon Y \to X$ be a $!$-able map in $\cC$. Then the category $\cD(Y)$ is internally smooth over $\cD(X)$ provided that $f$ is $\cD$-prim and $\Delta_f$ is $\cD$-suave.
\end{corollary}
\begin{proof}
    We verify the hypotheses of \Cref{prop: smoothness in a 6 ff}. Since $f$ is $\cD$-prim, $\mathrm{pr}$ is also $\cD$-prim (see \cite[Lemma 4.4.8]{heyer20246functorformalismssmoothrepresentations}).
    Invoking \Cref{prop: prima and suave maps}, we obtain two equivalences:
    \[
    \mathrm{pr}_{!}(\delta_{\mathrm{pr}}\otimes -) \simeq \mathrm{pr}_{\ast}(-) \quad \text{and} \quad f_{!}(\delta_{f}\otimes-) \simeq  f_{\ast}(-).
    \]
    Therefore, both $f_*(-)$ and $\mathrm{pr}_{*}(-)$ preserve colimits and are $\cD(X)$-linear by the projection formula.
    Using the $\cD(Y)$-linearity of $\Delta_{f,!}$, we obtain an equivalence
    \[
    \hom^{\cD(Y)}_{\cD(Y\times_{X}Y)}(\Delta_{f,!}\boldone, -) \simeq \Delta_{f}^{!}(-),
    \]
    since both functors are right adjoints to $\Delta_{f,!}(-)$. Furthermore, by \Cref{prop: prima and suave maps}, we can identify the internal hom functor with $\omega_{\Delta_f}\otimes\Delta_{f}^{\ast}(-)$, which is a colimit-preserving $\cD(Y)$-linear functor. This implies that $\Delta_{f,!}\boldone$ is $\mathscr{D}(Y)$-atomic. 
\end{proof}

\begin{remark}\label{remark: right adjoint to coev}
    Let $\cC$ and $\cD$ be as in \Cref{assumptions 6 functor}, and let $f\colon Y \to X$ be a $!$-able map. As observed earlier via \Cref{eq: Delta_! is linear}, the coevaluation functor is inherently $\cD(X)$-linear. Consequently, the right adjoint to $\coev$ can be identified with the $\cD(X)$-enriched hom functor:
    \begin{equation}
        \coev^{R}(-)\simeq \hom^{\cD(X)}_{\cD(Y\times_{X}Y)}(\coev(\boldone_{X}),-).
    \end{equation}
    It follows that $\cD(Y)$ is internally smooth over $\cD(X)$ if and only if $\coev(\boldone_{X})$ is $\cD(X)$-atomic.
\end{remark}

In the rigid case, this remark specializes to the following result.
\begin{lemma}\label{cor: smooth categories 6ff rigid case}
    Let $\cC$ and $\cD$ be as in \Cref{assumptions 6 functor}. Let $Y \to X$ be a $!$-able map in $\cC$. If $\cD(X)$ is rigid, then the category $\cD(Y)$ is internally smooth over $\cD(X)$ if and only if $\Delta_{!}(\boldone_{Y})$ is compact.
\end{lemma}
\begin{proof}
    This follows directly by combining \Cref{remark: right adjoint to coev} and \Cref{prop: compact objects over rigid categories}.
\end{proof}

\subsection{Proper categories in a six-functor formalism}
By analogy with the previous section, we investigate here how to obtain a proper \cat from a six-functor formalism, as defined in \Cref{assumptions 6 functor}. We conclude by proving a criterion for obtaining $\mathscr{D}$-prim covers, which will be used in \Cref{Appendix 2}. 

\begin{proposition}\label{prop: proper categories in a 6 functor}
    Let $\cC$ and $\cD$ be as in \Cref{assumptions 6 functor}, and let $f\colon Y \to X$ be a $!$-able map in $\cC$. Then the category $\cD(Y)$ is proper over $\cD(X)$ (and internally proper) provided the following conditions are satisfied:
    \begin{enumerate}
        \item The functor $\Delta_{f,\ast}$ preserves colimits (and, moreover, is $\cD(X)$-linear).
        \item The functor $f^{!}$ preserves colimits (and, moreover, is $\cD(X)$-linear).
    \end{enumerate}
\end{proposition}

\begin{proof}
    This statement follows immediately from the explicit descriptions of the right adjoints to the evaluation and coevaluation functors provided in \Cref{eq: evaluation 6 functor formalism} and \Cref{eq: coevaluation 6 functor formalism}.
\end{proof}

Nevertheless, from the preceding proposition, we can deduce reasonable conditions under which proper \cat arise within a six-functor formalism.

\begin{corollary}\label{cor: proper cat in 6 ff}
    Let $\cC$ and $\cD$ be as in \Cref{assumptions 6 functor}, and let $f \colon Y \to X$ be a $!$-able map in $\cC$. Then the category $\cD(Y)$ is internally proper over $\cD(X)$ provided that $f$ is $\cD$-suave and $\Delta_f$ is $\cD$-prim.
\end{corollary}

\begin{proof}
    This follows from the natural equivalences
    $$
    \omega_f \otimes f^*(-) \xrightarrow{\sim} f^!(-)  \quad \text{and} \quad \Delta_{f,!}(\delta_{\Delta_f} \otimes -) \xrightarrow{\sim} \Delta_{f,\ast}(-),
    $$
    combined with the $\cD(Y)$-linearity of $\Delta_{f,!}(-)$, as established in \Cref{eq: Delta_! is linear}.
\end{proof}

\begin{remark}
    Another instance where the hypotheses of \Cref{prop: proper categories in a 6 functor} are satisfied occurs when considering rigid categories. Indeed, let $\cC$ and $\cD$ be as in \Cref{assumptions 6 functor}, and let $f \colon Y \to X$ be a $!$-able map in $\cC$. If $\cD(X)$ is rigid, then for $\cD(Y)$ to be internally proper over $\cD(X)$, it suffices for the right adjoint $\ev^{\mathrm{R}}(-)$ to preserve colimits.
\end{remark}

To conclude this section, we combine the obtained results to provide an explicit criterion for a \cat to be smooth and proper within a six-functor formalism.

\begin{proposition}\label{prop: smooth and proper cat in 6 ff}
    Let $\cC$ and $\cD$ be as in \Cref{assumptions 6 functor}, and let $f \colon Y \to X$ be a $!$-able map in $\cC$. Then the \cat $\cD(Y)$ is internally smooth and proper over $\cD(X)$ if $f$ is both $\cD$-suave and $\cD$-prim.
\end{proposition}

\begin{proof}
    By \cite[Lemma 4.5.9 (i)]{heyer20246functorformalismssmoothrepresentations}, if $f$ is $\cD$-suave (respectively, $\cD$-prim), then $\Delta_f$ is $\cD$-suave (respectively, $\cD$-prim). The statement then follows from \Cref{cor: smooth categories 6ff suave and prim case} and \Cref{cor: proper cat in 6 ff}.
\end{proof}

We now describe a criterion for obtaining $\mathscr{D}$-prim covers, which will be used in \Cref{Appendix 2}. We refer the reader to \cite{heyer20246functorformalismssmoothrepresentations} for the standard definition of the $\mathscr{D}$-topology (also called the $!$-topology) and for the notion of a $\mathscr{D}$-cover. In particular, a map $W \to X$ is a $\mathscr{D}$-cover if the functor $\mathscr{D}(-)$ satisfies universal $*$- and $!$-descent along $W \to X$.  

From now on, we assume the following:

\begin{assumptions}\label{assumpions 2}
    Let $\mathscr{C}$ and $\mathscr{D}$ be as in \Cref{assumptions 6 functor}. Assume, moreover, that $\mathscr{C}$ is the \cat of sheaves of $\mathrm{Ani}$ on an $\infty$-site $\mathscr{C}_{0}$, and that the six-functor formalism $\mathscr{D}$ arises as the extension of a stable, presentable six-functor formalism on $\mathscr{C}_0$, as explained in \cite{heyer20246functorformalismssmoothrepresentations}.
\end{assumptions}

\begin{proposition}\label{prop: D-prim cover}
    Let $\mathscr{C}$ and $\mathscr{D}$ be as in \Cref{assumpions 2}, and let $X$ be an object in $\mathscr{C}$ such that:
    \begin{enumerate}
        \item[(1)] There exists an affine object $W$ and an effective epimorphism $g\colon W \to X$ along which $\mathscr{D}(-)$ satisfies $*$-descent, and there exists an integer $n$ such that the \v{C}ech nerve of $g$ is an $n$-truncated diagram (i.e., the canonical map
    $$
    X\simeq \colim_{[m]\in\bDelta_{\leq n}} (W \stackrev{3} W\times_X W \stackrev{5 }W\times_{X}W\times_{X}W \stackrev{7}\cdots ) 
    $$
    is an equivalence).
    \item[(2)] The projection maps $W\times _{X} W \to W$ are $\mathscr{D}$-prim.
    \end{enumerate}
     Then $g\colon W \to X$ is a descendable $\mathscr{D}$-cover (i.e., $g_* \mathbb{1}$ is descendable). Moreover, $g$ is $\mathscr{D}$-prim.
\end{proposition}

\begin{proof}
    We begin by observing that since being prim is local on the target \cite[Lemma 4.5.7]{heyer20246functorformalismssmoothrepresentations}, it follows from hypothesis (2) that $g$ is $\mathscr{D}$-prim.
    
    First, we prove that $g$ is a descendable map. We need to verify that $g\colon W \to X$ is such that $g_* \mathbb{1} \in \mathscr{D}(X)$ is descendable (i.e., $\mathbb{1}_X$ belongs to the thick tensor ideal $\langle g_* \mathbb{1}_W \rangle$).
    
    Observe that by denoting the $n$-fold fiber product $W\times_X W \times_X \cdots \times_X W$ by $W_n$ and considering the maps
    $$
    g_n \colon  W_n \to W \to X,
    $$
    we obtain
    \begin{equation}\label{eq: limite 1}
        \mathbb{1}_X \simeq \lim_{[m] \in \bDelta_{\leq n}} g_{m,*} \mathbb{1}_{W_m}.
    \end{equation}

    We now show that $g_{m,*} \mathbb{1}_{W_m} \simeq g_* \mathbb{1}_W \otimes_X \dots \otimes_X g_* \mathbb{1}_W$. This follows by induction, starting with the base case $m=2$. We consider the following diagram: 
    $$
    \begin{tikzcd}
    & W \times_X W \arrow[ld, "pr_2"'] \arrow[rd, "pr_1"] & \\
    W \arrow[rd, "g"'] & & W \arrow[ld, "g"] \\
    & X. &
    \end{tikzcd}
    $$
    Since $g$ is prim, we have the following equivalences:
    $$
    \begin{aligned}
    g_{2,*} \mathbb{1}_{W_2} &\simeq g_* {pr_2}_* \mathbb{1}_{W \times_X W} \\ 
    &\simeq g_* ( {pr_2}_* pr_1^* \mathbb{1}_W ) \\
    &\simeq g_* ( \mathbb{1}_W \otimes {pr_2}_* pr_1^* \mathbb{1}_W ) \\
    &\simeq g_* ( \mathbb{1}_W \otimes g^* g_* \mathbb{1}_W ) \\
    &\simeq g_* \mathbb{1}_W \otimes g_* \mathbb{1}_W.
    \end{aligned}
    $$
    Here, the fourth equivalence follows from the base change property, and the last one from the projection formula. This computation implies that the finite limit in \Cref{eq: limite 1} belongs to the thick tensor ideal $\langle g_* \mathbb{1}_W \rangle$. Now, from \cite[Lemma 4.7.4]{heyer20246functorformalismssmoothrepresentations}, we deduce that $g$ is a $\mathscr{D}$-cover. 
\end{proof}

\begin{corollary}\label{cor: D-prim map}
    Let $\mathscr{C}$ and $\mathscr{D}$ be as in \Cref{assumpions 2}, and let $f\colon X \to Y$ be a map in $\mathscr{C}$.
    Assume the following conditions:
    \begin{enumerate}
        \item[(1)] The induced map $W \to Y$ is $\mathscr{D}$-prim.
        \item[(2)] There exists an affine object $W$ and an effective epimorphism $g\colon W \to X$ along which $\mathscr{D}(-)$ satisfies $*$-descent, and there exists an integer $n$ such that the \v{C}ech nerve of $g$ is an $n$-truncated diagram (i.e., the canonical map
        $$
        X\simeq \colim_{[m]\in\bDelta_{\leq n}} (W \stackrev{3} W\times_X W \stackrev{5 }W\times_{X}W\times_{X}W \stackrev{7}\cdots ) 
        $$
        is an equivalence).
        \item[(3)] The projection maps $W\times _{X} W \to W$ are $\mathscr{D}$-prim.
    \end{enumerate}
    Then $f$ is $\mathscr{D}$-prim.
\end{corollary}

\begin{proof}
    The condition that $f\colon X \to Y$ is $\mathscr{D}$-prim can be checked locally on $Y$ (see \cite[Lemma 4.5.7]{heyer20246functorformalismssmoothrepresentations}); in particular, we may assume $Y$ is affine. Moreover, by \Cref{prop: D-prim cover} and \cite[Corollary 4.7.5]{heyer20246functorformalismssmoothrepresentations}, it suffices to show that the composite map $W \to Y$ is $\mathscr{D}$-prim. This follows directly from condition (1).
\end{proof}

\begin{remark}
    We observe that condition (1) in \Cref{cor: D-prim map} is implied by the following condition:
    \begin{itemize}
        \item[(1')] For every affine object $T$ equipped with a map $T \to Y$, the pullback map $W\times_Y T \to T$ is $\mathscr{D}$-prim, and $Y$ is semi-separated (i.e., it has an affine diagonal).
    \end{itemize}
\end{remark}

We note that although the hypotheses of \Cref{cor: D-prim map}, especially condition (2), appear stringent, they are satisfied in various situations, such as the case considered in \Cref{Appendix 2}.

\section{Some Remarks on Condensed Mathematics}\label{Section 2}
\subsection{Analytic rings}
In this subsection, we review the definition of analytic rings as presented in \cite{Clausen_Scholze_lectures}, with a primary focus on those arising from rigid geometry. Throughout this section, we work over a complete non-archimedean base field $k$ of characteristic $0$, with a pseudo-uniformizer $\omega$ and ring of integers $\cO_{k}$.

The main objective of this section is to introduce the notation used in subsequent parts of this paper, such as the definition of \emph{animated affinoid $k$-algebras} as given in \cite{mikami2023fppfdescentcondensedanimatedrings}, and to describe their tensor product (see \Cref{prop: animated affinoid stable under products and tensor products}). We begin by introducing some basic notions related to condensed mathematics and analytic stacks, as explained in \cite{Clausen_Scholze_lectures}.

We denote by $\Mod_{\mathbb{Z}}^{\mathrm{cond}}$ the unbounded derived \cat of condensed abelian groups. We say that $\mathscr{A}$ is a \emph{condensed $\bbE_\infty$-algebra} if $\mathscr{A}$ is an $\bbE_\infty$-algebra in $\Mod^{\mathrm{cond}}_\Z$. In this case, we denote by $\Mod^{\mathrm{cond}}_\mathscr{A}$ the \cat of $\mathscr{A}$-modules in $\Mod^{\mathrm{cond}}_\Z$.

\begin{definition}
\label{def:analytic_ring}
    An analytic ring is a pair $A=(\cA, \cD(A))$, where $\cA$ is a condensed $\bbE_{\infty}$-algebra and $\cD(A)$ is a full subcategory of $\Mod^{\mathrm{cond}}_{\cA}$ satisfying the following properties:
    \begin{enumerate}
        \item The \cat $\cD(A)$ contains the object $\cA$.
        \item \label{def:completion} Let $\iota_{A}\colon \cD(A) \hookrightarrow \Mod_{\cA}^{\mathrm{cond}}$ denote the natural inclusion functor. Then $\iota_A$ commutes with all limits and colimits, and it admits a left adjoint (called the \emph{completion functor along $A$}), which we denote by $-\otimes_{\cA}A$. The composition $\iota_A \circ (-\otimes_{\cA}A)$ preserves connective objects.
        \item For every object $M$ in $\cD(A)$ and every $N$ in $\Mod^{\mathrm{cond}}_{\Z}$, the internal mapping object $\uline{\Hom}_{\Z}(N,M)$ lies in $\cD(A)$. 
    \end{enumerate}
\end{definition}

\begin{remark} \cite[Remark 2.2.3]{montagnani2025axiomaticapproachanalytic1affineness} 
    By definition, the essential image of $\iota_{A}$ is a localization of $\Mod_{\cA}^{\mathrm{cond}}$. In particular, since $\Mod_{\cA}^{\mathrm{cond}}$ is presentable and stable, it follows that $\cD(A)$ is presentable and stable as well. Moreover, $\cD(A)$ admits a symmetric monoidal structure obtained by completing the natural tensor product of $\Mod_{\cA}^{\mathrm{cond}}$ along $A$.
\end{remark}

In this paper, we are mainly interested in analytic rings arising from rigid analytic geometry and affinoid algebras (i.e., quotients of Tate algebras by a finitely generated ideal). Recall that to every complete Huber pair $(R,R^{+})$, we can associate an analytic ring 
\[
(R,R^+)_{\Solid} \coloneqq (R^{\triangleright},\cD(R)_{\Solid})
\]
defined as follows.

We define $R^{\triangleright}$ to be the condensed ring associated with $R$ via its given topology. We then consider the following construction, as explained in \cite{andreychev2021pseudocoherent}.
For $f \in R^+$ and $M \in \Mod_{R}^{\mathrm{cond}}$, we say that $M$ is \emph{$f$-solid} if
\[
1 - f\sigma \colon \uline{\Hom}_\Z(\mathrm{Null}_\Z, M) \longrightarrow \uline{\Hom}_\Z(\mathrm{Null}_\Z, M)
\]
is an equivalence. The inclusion of the full sub-$\infty$-category of $f$-solid $R$-modules admits a left adjoint $(-)^{f\Solid}$, called \emph{$f$-solidification}. We define 
\[
\cD(R)_\Solid \subseteq \Mod_{R}^{\mathrm{cond}}
\]
to be the full sub-$\infty$-category of condensed $R$-modules that are $f$-solid for all $f \in R^+$. This construction defines an analytic ring, as shown in \cite[Theorem 3.28]{andreychev2021pseudocoherent}. Furthermore, using \cite[Theorem 3.33]{andreychev2021pseudocoherent}, we see that this construction determines a fully faithful embedding from the category of complete Huber pairs into the category of analytic rings. In particular, when the Huber pair arises from an affinoid algebra $A$, we denote the associated analytic ring by 
\[
A_\Solid \coloneqq (A,A^\circ)_\Solid \coloneqq (A^\triangleright, \cD(A)_\Solid).
\]
The objects in $\cD(A)_\Solid$ are called $A^\circ$-solid modules. We denote the symmetric monoidal structure of $\cD(A)_\Solid$ by $-\otimes^{\Solid}_{A}-$.
\begin{remark}
    We note that, by \cite[Proposition 2.3]{mikami2023fppfdescentcondensedanimatedrings}, Tate algebras over $k$ and, more generally, affinoid $k$-algebras lie in $\cD(k)_{\Solid}$.
\end{remark}

The following is an analogue of \cite[Definition 5.8]{andreychev2021pseudocoherent}.

\begin{definition}\label{def: solidification functor}
    We define the \emph{$k$-solidification functor} to be the functor 
    \[
    \uline{(-)}\colon \Mod_{k} \to \cD(k)_{\Solid}
    \]
    determined by mapping $k$ to $k^{\triangleright}$ (endowed with its non-archimedean topology) and requiring that it commutes with colimits. The objects in the image of this functor are called \emph{discrete $k$-solid modules}. 
\end{definition}

Note that, although the term ``discrete'' appears in the nomenclature, the objects in the image of the functor defined in \Cref{def: solidification functor} do not necessarily carry the discrete topology. Furthermore, this functor is by definition a left adjoint; we denote its right adjoint by
\[
(-)(\ast)\colon \cD(k)_{\Solid} \to \Mod_{k}.
\] 
Similarly, for every affinoid $k$-algebra $A$, we can define an adjunction
\[
\begin{tikzcd}
	{\Mod_{A^{\triangleright}(\ast)}} && {\cD(A)_{\Solid}}
	\arrow[""{name=0, anchor=center, inner sep=0}, "{\uline{(-)}}"', shift right=2, from=1-1, to=1-3]
	\arrow[""{name=1, anchor=center, inner sep=0}, "{(-)(\ast)}"', shift right=2, from=1-3, to=1-1]
	\arrow["\dashv"{anchor=center, rotate=90}, draw=none, from=0, to=1]
\end{tikzcd}
\]
The objects in the image of $\uline{(-)}$ are referred to as \emph{discrete $A$-solid modules}.

\begin{lemma}\label{lem: tensor product of discete modules}
    Let $\uline{A}$ and $\uline{B}$ be two discrete $k$-solid modules. Then the tensor product $\uline{A} \otimes^{\Solid}_{k} \uline{B}$ is equivalent to the discrete $k$-solid module associated with the $k$-linear tensor product of $A$ and $B$.
\end{lemma}

\begin{proof}
    It suffices to observe that every object in $\Mod_{k}$ can be written as a colimit of copies of $k$. Since the $k$-solidification functor commutes with colimits and sends $k$ to $k^{\triangleright}$ in $\cD(k)_{\Solid}$, we deduce that
    \[
    \uline{A} \otimes^{\Solid}_{k} \uline{B} \simeq \uline{A \otimes_{k} B}.
    \]
\end{proof}
  
\subsection{Animated affinoid algebras}
In this section, we review the definition of condensed derived affinoid $k$-algebras, following \cite{mikami2023fppfdescentcondensedanimatedrings}. 
Throughout this section, we work over the analytic ring $(k,\Z)_\Solid$, as is also done in \cite{mikami2023fppfdescentcondensedanimatedrings}. However, the essential requirement for the arguments in the subsequent subsection is a tensor product that recovers the classical completed tensor product for affinoid algebras, as explained in \Cref{rmk: completed tensor}. Any other choice of base analytic ring satisfying this property, such as $k_\Solid$, functions equally well. 

In this subsection, we denote by $\otimes^{\Solid}_k$ the tensor product in $\cD(k,\Z)_\Solid$.

\begin{definition}\label{def:nuclear}
Let $\cC$ be a stable closed symmetric monoidal $\infty$-category. An object $X \in \cC$ is called \emph{nuclear} if, for every compact object $P \in \cC$, the canonical map
\[
 \Map_{\cC}\!\bigl(\boldone_{\cC},\, \uline{\Hom}_{\cC}(P,\boldone_{\cC}) \otimes_{\cC} X \bigr) 
  \;\longrightarrow\;
 \Map_{\cC}(P,X)
\]
is an equivalence of anima. The full subcategory of $\cC$ spanned by nuclear objects is denoted by $\Nuc(\cC) \subseteq \cC$.
\end{definition}

Let $A$ be an analytic ring and let $\cD(A)$ be its associated \cat. The $\infty$-category $\Nuc(A)$ is a presentable, stable, full sub-$\infty$-category of $\cD(A)$ that is closed under colimits and tensor products \cite[Theorem 8.6]{complex}. This closure property equips $\Nuc(A)$ with a symmetric monoidal structure inherited from $\cD(A)$.

\begin{definition}\cite[Definition 2.11]{mikami2023fppfdescentcondensedanimatedrings}\label{def: derived affinoid}
    An \emph{animated affinoid $k$-algebra} $A$ is a condensed animated $k$-algebra such that: 
    \begin{enumerate}
        \item $A$ is a $\Z$-solid $k$-algebra and, moreover, $A$ is nuclear over $(k, \Z)_\Solid$.
        \item $\pi_{0}(A)$ is an underived condensed affinoid algebra over $k$ (i.e., there exists a classical affinoid $k$-algebra $B$ such that, if we denote by $\uline{B}$ the algebra $B$ regarded as a condensed $k$-module, then there is an equivalence $\pi_{0}(A) \simeq \uline{B}$).
    \end{enumerate}
\end{definition}

\begin{remark}\label{rmk: animated affinoid nuclear over k}
    We observe that the first condition of \Cref{def: derived affinoid} is equivalent to requiring $A$ to be nuclear over $k_\Solid$. Indeed, by \cite[Corollary 2.7]{mikami2023fppfdescentcondensedanimatedrings}, an animated affinoid $k$-algebra is nuclear over $k_\Solid$ if and only if it is nuclear over $(k,\Z)_\Solid$.
\end{remark}

We denote by $\mathrm{dAfdAlg}_{k}$ the \cat of animated affinoid $k$-algebras, regarded as a full subcategory of $\cD(k,\Z)_{\Solid}$. 

\begin{proposition}\label{prop: animated affinoid stable under products and tensor products}
    The \cat $\mathrm{dAfdAlg}_{k}$ is stable under finite products and tensor products.
\end{proposition}
\begin{proof}
    The fact that the tensor product of animated affinoid $k$-algebras yields an animated affinoid $k$-algebra is proven in \cite[Proposition 2.14]{mikami2023fppfdescentcondensedanimatedrings}. To verify that the \cat $\mathrm{dAfdAlg}_{k}$ is stable under finite products, we observe that $\pi_0(-)$ commutes with products, and that nuclear categories are stable under finite limits.
\end{proof}

\begin{remark}\label{rmk: completed tensor}
As explained in the proof of \cite[Proposition 2.14]{mikami2023fppfdescentcondensedanimatedrings}, the tensor product of derived affinoid $k$-algebras recovers the completed tensor product of classical affinoid algebras. In particular, we have 
\[
\pi_0(A\otimes^{\Solid}_B C) \simeq \pi_0A\widehat{\otimes}_{\pi_{0}B} \pi_0{C}.
\]
\end{remark}
Motivated by this property, we denote the tensor product in $\mathrm{dAfdAlg}_{k}$ by $-\widehat{\otimes}-$. In particular, it can be regarded as a tensor product between animated affinoid algebras in $\Nuc(k)$.

\subsection{Nuclear categories for animated affinoid algebras}

By \cite[Kollar 3.18]{andreychev2023ktheorieadischerraume}, we deduce an equivalence between $\Nuc(\cD(k)_\Solid)$ and $\Nuc(\cD(k,\Z)_\Solid)$. For this reason, hereafter we will simply write $\Nuc(k)$.

For animated affinoid $k$-algebras in $\Nuc(k)$, we consider the tensor product described above, denoted by $-\widehat{\otimes}-$.

If $A$ is an animated affinoid $k$-algebra, by \cite[Proposition 12.21]{analytic} and using the analytic ring structure on $\pi_{0}(A)$ (as described in \cite[Theorem 3.28]{andreychev2021pseudocoherent}), we can associate to it an analytic ring
\begin{equation*}\label{eq: analytic ring structure on derived affinoid}
    A_{\Solid} \coloneqq (A, \cD(A)_{\Solid}).
\end{equation*} 
In particular, an object $M \in \Mod_{A}^{\mathrm{cond}}$ lies in $\cD(A)_\Solid$ if and only if $\pi_{i}(M)$ belongs to $\cD(\pi_{0}(A))_{\Solid}$ for every integer $i$. 

In this way, as explained in \cite[Proposition 2.18]{mikami2023fppfdescentcondensedanimatedrings}, we obtain a fully faithful embedding:
\[\begin{tikzcd}[row sep=0.05cm, column sep=1.5cm]
	{\mathrm{dAdfdAlg}_k} && {\mathrm{AnRing}_{(k,\Z)_\Solid}} \\
	A && {A_\Solid}.
	\arrow[from=1-1, to=1-3]
	\arrow[maps to, from=2-1, to=2-3]
\end{tikzcd}\]

For an animated affinoid $k$-algebra $A$, we will hereafter denote by $\Nuc(A)$ the \cat $\Nuc(\cD(A)_\Solid)$ of nuclear objects in $\cD(A)_\Solid$. The following proposition provides an explicit description of this \cat.

\begin{proposition}\label{prop: nuclear categories as module categories}
    Let $A$ be an animated affinoid $k$-algebra. Then there is an equivalence 
    \[
    \Nuc(A) \simeq \Mod_{A}(\Nuc(k)).
    \]
    In particular, $\Nuc(A)$ is rigid.
\end{proposition}
\begin{proof}
    This is a special case of \cite[Lemma 4.4.7]{wang2026relativegagatheoremapplication}. We apply the cited lemma to the canonical map of analytic rings 
    \[
    (k,\Z)_\Solid \to A_\Solid.
    \]
    Indeed, by \Cref{def: derived affinoid} and \Cref{rmk: animated affinoid nuclear over k}, $A$ belongs to $\Nuc(k)$. Moreover, for every profinite set $S$, the condensed $k$-module of continuous functions $\uline{\mathscr{C}(S,k)}$ lies in $\Nuc(k)$ by \cite[Kollar 3.14]{andreychev2023ktheorieadischerraume}.
\end{proof}

From the above proposition, we immediately deduce the following corollary.

\begin{corollary}\label{cor: relative nuclearity affinoid case}
    If $A \to B$ is a map of animated affinoid $k$-algebras, then $B$ is nuclear over $A$. 
\end{corollary}
\begin{proof}
    Since $B$ can be viewed as an $A$-module, the statement reduces, by \Cref{prop: nuclear categories as module categories}, to showing that $B$ is nuclear over $k$. This follows directly from \Cref{rmk: animated affinoid nuclear over k}.
\end{proof}

\begin{corollary}\label{Kunneth formula nuclear affine case}
    Let $A$, $B$, and $C$ be animated affinoid $k$-algebras. Then there is an equivalence 
    \[
    \Nuc(A) \otimes_{\Nuc(B)} \Nuc(C) \simeq \Nuc(A \widehat{\otimes}_{B} C).
    \]
\end{corollary}
\begin{proof}
    By \Cref{prop: animated affinoid stable under products and tensor products}, the tensor product $A \widehat{\otimes}_B C$ is an animated affinoid algebra. The statement then easily follows from the explicit description of nuclear \cates provided in \Cref{prop: nuclear categories as module categories}. 
\end{proof}

\begin{lemma}\label{lem: pullback derived affinoids}
    Let $f \colon A \to B$ be a map of animated affinoid $k$-algebras. Then the pullback functor 
    \[
    f^{\ast} \colon \cD(A)_{\Solid} \to \cD(B)_{\Solid}
    \]
    preserves nuclear modules. Moreover, under the equivalence of \Cref{prop: nuclear categories as module categories}, the induced map
    \[
    f^{\ast} \colon \Nuc(A) \to \Nuc(B)
    \]
    can be identified with the tensor product $-\widehat{\otimes}_{A} B$.
\end{lemma}
\begin{proof}
    From \cite[Theorem 2.9 (c)]{meyer2025qhodgecomplexesrefinedoperatornametc}, we deduce that the pullback functor $f^{\ast}$ restricts to the nuclear subcategories. The latter statement follows from the exact same argument used in \cite[Proposition 3.2.8]{montagnani2025axiomaticapproachanalytic1affineness}, which extends verbatim to our setting.
\end{proof}

\section{Nuclear categories for derived rigid analytic spaces}\label{Section 3}
In this section, we introduce the notion of derived rigid analytic spaces, which will be employed in the final section. A derived rigid space will be described via the gluing of animated affinoid $k$-algebras along rational localizations. We will then consider quasi-compact semi-separated derived rigid analytic spaces and their nuclear \cates, proving a K\"{u}nneth formula (\Cref{prop: kunneth for nuclear moduels}) and establishing the existence of a $6$-functor formalism (\Cref{thm: 6 functor for nuclear modules in rigid geometry}). 

\subsection{Rational localization and derived rigid analytic spaces}
In this subsection, we define condensed derived rigid analytic varieties as the gluing of animated affinoid $k$-algebras along rational localizations. A similar approach has been carried out in the bornological setting by Aron Soor in \cite{soor2025sixfunctorformalismquasicoherentsheaves}.

As explained in \cite[Lemma 2.24]{mikami2023fppfdescentcondensedanimatedrings}, let $A$ be a classical affinoid algebra, and let $f_{1},\dots, f_{n}, g$ be elements that generate the unit ideal of $A$. We can view $A$ as a $\Z[T]$-algebra via the map sending $T$ to $g$. Under this $\Z[T]$-module structure, we have the equivalence: 
\begin{equation*}
    A \otimes_{(\Z[T],\Z)_{\Solid}} (\Z[T,T^{-1}], \Z)_{\Solid} \simeq A[g^{-1}].
\end{equation*}
Moreover, we can consider $A[g^{-1}]$ as an algebra over $\Z[U_{1},\dots, U_{n}]$ via the map sending $U_{i}$ to $f_{i}/g$. This yields an isomorphism: 
\begin{equation*}
    (A \otimes_{(\Z[T],\Z)_{\Solid}} (\Z[T,T^{-1}], \Z)_{\Solid})\otimes_{(\Z[U_{1},\dots, U_{n}],\Z)_{\Solid}} \Z[U_{1},\dots, U_{n}]_{\Solid} \simeq A\left\langle \frac{f_1}{g}, \dots, \frac{f_n}{g} \right\rangle.
\end{equation*}
Here, $A\left\langle \frac{f_1}{g}, \dots, \frac{f_n}{g} \right\rangle$ denotes the usual rational localization of $A$. 

When $A$ is a derived affinoid $k$-algebra, we define: 
\begin{equation}\label{eq: rational loc}
    A\left\langle \frac{f_1}{g}, \dots, \frac{f_n}{g} \right\rangle \coloneqq
    (A \otimes_{(\Z[T],\Z)_{\Solid}} (\Z[T,T^{-1}], \Z)_{\Solid}) \otimes_{(\Z[U_{1},\dots, U_{n}],\Z)_{\Solid}} \Z[U_{1},\dots, U_{n}]_{\Solid}.
\end{equation}
This leads to the following proposition:

\begin{proposition}\cite[Proposition 2.25]{mikami2023fppfdescentcondensedanimatedrings}
    If $A$ is an animated affinoid $k$-algebra, then $A\left\langle \frac{f_1}{g}, \dots, \frac{f_n}{g}\right\rangle$ is also an animated affinoid $k$-algebra such that
\[ \pi_0 \left(A\left\langle \frac{f_1}{g}, \dots, \frac{f_n}{g}\right\rangle\right) \simeq (\pi_0 A)\left\langle \frac{f_1}{g}, \dots, \frac{f_n}{g} \right\rangle. \]
\end{proposition}

\begin{definition}\cite[Definition 2.26]{mikami2023fppfdescentcondensedanimatedrings}
    Let $A$ be a derived affinoid $k$-algebra. A rational localization of $A$ is a map of animated affinoid $k$-algebras   
    \begin{equation*}
        A \to A\left\langle \frac{f_1}{g}, \dots, \frac{f_n}{g} \right\rangle
    \end{equation*}
    induced by \Cref{eq: rational loc}.
\end{definition}

If $A$ is an animated affinoid $k$-algebra, we denote by $\mathrm{Rat}_{A}$ the \cat of rational localizations of $A$.

\begin{remark}\label{rmk: rational loc on pi0}
    By \cite[Proposition 2.29]{mikami2023fppfdescentcondensedanimatedrings}, there is an equivalence of categories from $\mathrm{Rat}_{A}$ to $\mathrm{Rat}_{\pi_{0}(A)}$. This category, as explained above, coincides with the category of usual rational localizations of the affinoid $k$-algebra $\pi_{0}A(\ast)$.
\end{remark}

\begin{definition}
    We say that a map $A \to B$ of animated affinoid $k$-algebras is a \emph{homotopy epimorphism} if the induced map
    \[
    B\widehat{\otimes}_{A}B \to B
    \]
    is an equivalence.
\end{definition}

\begin{lemma}
    The class of rational localizations is stable under pushouts and compositions. Moreover, every rational localization is a homotopy epimorphism.
\end{lemma}
\begin{proof}
Let $A \to B$ and $A \to C$ be maps of derived affinoid $k$-algebras, with $A \to B$ a rational localization. To prove that $C \to B\widehat{\otimes}_A C$ is a rational localization, using \cite[Proposition 2.29]{mikami2023fppfdescentcondensedanimatedrings}, it suffices to show that the induced map on $\pi_0$ is a classical rational localization.

Using the equivalence
    \[
    \pi_{0}(B\otimes^{\Solid}_{A}C) \simeq \pi_{0}(B)\otimes^{\Solid}_{\pi_{0}(A)} \pi_{0}(C),
    \]
    and the fact that the tensor product can be identified with the completed tensor product, we deduce that the induced map 
    \[
    \pi_{0}(C) \to \pi_{0}(B\otimes^{\Solid}_{A}C) \simeq \pi_{0}(B)\otimes^{\Solid}_{\pi_{0}(A)} \pi_{0}(C)
    \]
    can be identified with a rational localization.

The fact that a rational localization is a homotopy epimorphism follows from \cite[Lemma 2.27]{mikami2023fppfdescentcondensedanimatedrings} and \cite[Exercise 12.17]{analytic}. 

\begin{remark}
    In particular, this lemma implies that a rational localization is a Zariski morphism in the opposite \cat $\mathrm{dAfdAlg}^{\mathrm{op}}$, as defined in \cite[Definition 1.3.4]{montagnani2025axiomaticapproachanalytic1affineness}. This property will be utilized later. 
\end{remark}
    
To conclude the proof, by invoking \cite[Proposition 2.29]{mikami2023fppfdescentcondensedanimatedrings}, we deduce that rational localizations are stable under compositions, since classical rational localizations are stable under compositions.
\end{proof}

We denote by $\mathrm{dAfd}_k$ the opposite \cat of $\mathrm{dAfdAlg}_k$, and we denote by $\dSp(A)$ the object in $\mathrm{dAfd}_k$ associated with the animated affinoid $k$-algebra $A$. We refer to the objects in the \cat $\mathrm{dAfd}_k$ as \emph{derived $k$-affinoid spaces}.

\begin{definition}\cite[Definition 4.12]{soor2025sixfunctorformalismquasicoherentsheaves}\label{def: weak analytic site}
Let $X \coloneqq \dSp(A)$ be a derived $k$-affinoid space.
\begin{enumerate}
    \item[(i)] The \emph{small analytic site} of $X$ is the $\infty$-site with:
    \begin{enumerate}
        \item[(a)] underlying $\infty$-category given by the full subcategory $\mathrm{Rat}_X$ of $\mathrm{dAfd}_{/X}$ spanned by rational subdomains $\Sp(B) \to \Sp(A)$;
        \item[(b)] covering sieves generated by finite families of rational subdomains 
        \[
        \{\dSp(B_i) \to \dSp(B)\}_{i=1}^n
        \]
        such that 
        \[
        \{\Sp(\pi_0(B_i)(\ast)) \to \Sp(\pi_0(B)(\ast))\}_{i=1}^n
        \]
        is an analytic cover of the classical rigid space $\Sp(\pi_0(B)(\ast))$.
    \end{enumerate}
    
    \item[(ii)] The \emph{big analytic site} on $\mathrm{dAfd}$ is the $\infty$-site with:
    \begin{enumerate}
        \item[(a)] underlying $\infty$-category given by $\mathrm{dAfd}$;
        \item[(b)] covering sieves generated by finite families of rational subdomains 
        \[
        \{\dSp(B_i) \to \dSp(B)\}_{i=1}^n
        \]
         such that 
         \[
        \{\Sp(\pi_0(B_i)(\ast)) \to \Sp(\pi_0(B)(\ast))\}_{i=1}^n 
         \]
        is an analytic cover of the classical rigid space $\Sp(\pi_0(B)(\ast))$.
    \end{enumerate}
\end{enumerate}
\end{definition}

\begin{definition}
We define
\begin{equation}
    \mathrm{Shv}(\mathrm{dAfd}) \coloneqq \mathrm{Shv}(\mathrm{dAfd}, \mathrm{Ani})
\end{equation}
to be the $\infty$-category of sheaves on the $\infty$-site $\mathrm{dAfd}$ equipped with the analytic topology.
\end{definition}

\begin{remark}\label{rmk: finite cover}
    Since every classical affinoid space $\Sp(A)$ is quasi-compact, from \Cref{def: weak analytic site} and \Cref{rmk: rational loc on pi0} we immediately deduce that every analytic cover of $\Sp(A)$ can be refined to a finite cover. 
\end{remark}

\begin{lemma}\label{lem: Nuc is a sheaf}
The following properties hold.
\begin{enumerate}
    \item[(i)] The functor $\Nuc \colon \mathrm{dAfd}^{\mathrm{op}} \to \mathrm{CAlg}(\mathrm{Pr}^L_{\mathrm{st}})$ is a sheaf on $\mathrm{dAfd}$ in the analytic topology.
    \item[(ii)] For every animated affinoid $k$-algebra $A$, the functor represented by $\mathrm{dSp}(A)$,
    \[
    \mathrm{dSp}(A) \colon \mathrm{dAfd}^{\mathrm{op}} \to \mathrm{Ani},
    \]
    is a sheaf in the analytic topology.
\end{enumerate}
\end{lemma}
\begin{proof}
    To prove (i), we observe that by \Cref{rmk: finite cover} we can reduce the problem to considering a finite cover
    \[
    \{B \to A_i\}_{i\in I}
    \]
    of animated affinoid $k$-algebras. 
    By \Cref{prop: nuclear categories as module categories} and \cite[Lemma D.3.5.5]{SAG}, the functor 
    \[
    \Nuc(-)\simeq \Mod_{(-)}(\Nuc(k))\colon \mathrm{dAfd}^{\mathrm{op}} \to \LPr_{\mathrm{st}}
    \]
    commutes with finite products. Under this observation, using \cite[Proposition 3.22]{mathew2016galoisgroupstablehomotopy}, we can reduce the problem to showing that 
    \[
    C \coloneqq \prod_{i \in I} A_{i}
    \]
    is a descendable algebra in $\Nuc(B)$. 
    First, we observe that by \Cref{prop: animated affinoid stable under products and tensor products}, $C$ is an animated affinoid $k$-algebra. The statement then follows from \cite[Theorem 4.15]{mikami2023fppfdescentcondensedanimatedrings}, observing that by \Cref{cor: relative nuclearity affinoid case}, $C$ lies in $\Nuc(B)$. 
    
    We now prove that the functor 
\[
\mathrm{Map}_{\mathrm{dAfdAlg}_k}(A,-)\colon \mathrm{dAfdAlg}_k \to \Ani
\]
is a sheaf in the analytic topology. Consider an animated affinoid $k$-algebra $B$ and a cover as described above. As already explained, $C$ is descendable in $\Nuc(B)$; therefore, by \cite[Proposition 3.20]{mathew2016galoisgroupstablehomotopy}, we can realize $B$ as the following limit in $\Nuc(B)\simeq \Mod_{B}(\Nuc(k))$:
    \[
   \begin{tikzcd}
	B & {\lim(C } & {C\otimes_{\Nuc(B)}C} & {\dots).}
	\arrow["{\simeq }", from=1-1, to=1-2]
	\arrow[shift right, from=1-2, to=1-3]
	\arrow[shift left, from=1-2, to=1-3]
	\arrow[from=1-3, to=1-4]
	\arrow[shift left=3, from=1-3, to=1-4]
	\arrow[shift right=3, from=1-3, to=1-4]
\end{tikzcd}
    \]
    Now, since $\Map_{\mathrm{dAfdAlg}_k}(A,-)$ commutes with limits, the statement follows from \cite[Proposition 4.6.2.17]{HA}.
\end{proof}

\begin{definition}\cite[Definition 4.19]{soor2025sixfunctorformalismquasicoherentsheaves}\label{def: derived rigid variety}
\leavevmode
\begin{enumerate}
    \item[(i)] A \emph{derived $k$-affinoid space} is an object of $\mathrm{Shv}(\mathrm{dAfd})$ that is isomorphic to $\mathrm{dSp}(A)$ for some $A \in \mathrm{dAfdAlg}_{k}$.
    
    \item[(ii)] Let $X \coloneqq \mathrm{dSp}(A)$ be a derived $k$-affinoid space. An \emph{analytic subspace} $U \hookrightarrow X$ is a subsheaf $U$ of $X$ such that:
    \begin{itemize}
        \item[$\star$] There exists a small collection $\{\mathrm{dSp}(A_i)\}_{i \in I}$ of derived affinoid spaces and an effective epimorphism $\coprod_{i \in I} \mathrm{dSp}(A_i) \to U$ in $\mathrm{Shv}_{\mathrm{weak}}(\mathrm{dAfd})$ such that each map $\mathrm{dSp}(A_i) \to \mathrm{dSp}(A)$ is a rational subdomain.
    \end{itemize}

    \item[(iii)] Let $X \in \mathrm{Shv}(\mathrm{dAfd})$. An \emph{analytic subspace} $Y \hookrightarrow X$ is a subsheaf such that for every derived affinoid space $\mathrm{dSp}(A) \to X$ mapping to $X$, the pullback $Y \times_X \mathrm{dSp}(A) \hookrightarrow \mathrm{dSp}(A)$ is an analytic subspace in the sense of (ii). The morphism $Y \to X$ is then termed an \emph{open immersion}.

    \item[(iv)] A \emph{derived rigid analytic space} is an object $X \in \mathrm{Shv}(\mathrm{dAfd})$ such that there exists a small collection $\{\mathrm{dSp}(A_i)\}_{i \in I}$ of derived affinoid subspaces for which $\coprod_{i \in I} \mathrm{dSp}(A_i) \to X$ is an effective epimorphism. We denote the full subcategory of $\mathrm{Shv}(\mathrm{dAfd})$ consisting of derived rigid spaces by $\mathrm{dRig}$.
\end{enumerate}
\end{definition}

We remark that the category $\mathrm{dRig}$ admits all fiber products and colimits. The terminal object is $\mathrm{Sp}(k)$. Moreover, every classical rigid analytic variety can be viewed as a derived rigid analytic space; we reserve the term \emph{rigid analytic varieties} exclusively for those derived rigid analytic spaces that arise from classical rigid analytic varieties.

\begin{definition}
    Let $f \colon X \to Y$ be a map of derived rigid analytic spaces. 
    \begin{enumerate}
        \item We say that $f$ is \emph{affine} if for every derived affinoid map $\dSp(A) \to Y$, the pullback $X \times_Y \dSp(A)$ is a derived affinoid space. 
        \item We say that $X$ is \emph{quasi-compact} if every affinoid cover of $X$ can be refined to a finite cover. 
        \item We say that $X$ is \emph{semi-separated} if the diagonal map 
        \[
        X \to X\times_{\dSp(k)}X
        \]
        is affine. 
    \end{enumerate}
\end{definition}




\subsection{Nuclear categories over derived rigid analytic spaces}

Let $X$ be a derived rigid analytic space over $\mathrm{Sp}(k)$. We define the \emph{nuclear category} of $X$, via Kan extension using \Cref{lem: Nuc is a sheaf}, as the following limit over a cover of $X$ consisting of derived affinoid spaces:
\begin{equation*}
    \Nuc(X) \coloneqq \Lim_{\dSp(A_{i})\to X}\Nuc(\dSp(A_{i})).
\end{equation*}
In this way, for every map $f\colon X \to Y$ of derived rigid analytic varieties, we obtain a pullback functor for nuclear categories
\begin{equation}\label{eq: pullback}
    f^{\ast}\colon \Nuc(Y) \to \Nuc(X)
\end{equation}
which extends the one introduced in \Cref{lem: pullback derived affinoids}.

In particular, given a derived rigid analytic space $X$, its nuclear category $\Nuc(X)$ is a presentable \cat, tensored over $\Nuc(k)$.

Consequently, given two rigid analytic varieties, we can define the tensor product of their associated nuclear categories as the tensor product in $\Mod_{\Nuc(k)}(\LPr)$. We will employ the following notation: let $X$ and $Y$ be two rigid analytic varieties over $\Spk$; we define the tensor product of their nuclear categories, denoted by $-\cotimes-$, as follows:
\begin{equation*}
    \Nuc(X)\cotimes \Nuc(Y) \coloneqq \Nuc(X)\otimes_{\Nuc(k)}\Nuc(Y),
\end{equation*}
where the term on the right-hand side denotes the tensor product in $\Mod_{\Nuc(k)}(\LPr)$.

\begin{proposition}\label{prop: Nuc(X) is rigid}
    Let $X$ be a quasi-compact and semi-separated derived rigid analytic space. Then $\Nuc(X)$ is a rigid \cat.
\end{proposition}
\begin{proof}
    Observe that our setup satisfies \cite[Assumption 1.3.10]{montagnani2025axiomaticapproachanalytic1affineness} by taking $\cC$ to be the \cat of derived affinoid spaces, and $\widehat{\cC}$ to be the \cat $\mathrm{Shv}(\mathrm{dAdf})$. Therefore, this is a special case of \cite[Proposition 1.3.17]{montagnani2025axiomaticapproachanalytic1affineness}.
\end{proof}

Using this tensor product, we obtain a Künneth formula for nuclear \cates.

\begin{proposition}\label{prop: kunneth for nuclear moduels}
    Let $X$ and $Y$ be two quasi-compact, semi-separated derived rigid analytic spaces over $\Spk$. Then we have the following equivalence of categories:
    \begin{equation*}
        \Nuc(X)\cotimes \Nuc(Y) \simeq \Nuc(X\times Y).
    \end{equation*}
\end{proposition}
\begin{proof}
    Consider an affinoid cover $\{\dSp(A_i)\}_{i}$ of the derived rigid analytic space $X$. We first prove that
    \begin{equation}\label{eq:Kunneth step 1}
        \Nuc(X \times \dSp(B)) \simeq \Nuc(X) \cotimes \Nuc(\dSp(B)).
    \end{equation}
    By descent, we can express $\Nuc(X)$ as $\Lim_{i}\Nuc(\dSp(A_{i}))$. Since $\Nuc(\dSp(B))$ is rigid, it is dualizable as a $\Nuc(k)$-linear \cat; hence, the functor $-\cotimes \Nuc(\dSp(B))$ commutes with limits. In particular, we obtain
    \begin{equation*}
        \Lim_{i}(\Nuc(\dSp(A_{i}))\cotimes \Nuc(\dSp(B))) \simeq \left(\Lim_{i}\Nuc(\dSp(A_{i}))\right)\cotimes \Nuc(\dSp(B)).
    \end{equation*}
    Since $\{\dSp(A_{i})\times \dSp(B) \}_{i}$ constitutes a cover of $X \times \dSp(B)$, we can again use descent for nuclear categories and the completed tensor product to identify $\Lim_{i}(\Nuc(\dSp(A_{i}))\cotimes \Nuc(\dSp(B)))$ with $\Nuc(X \times \dSp(B))$.
    
    To conclude, we observe that the same argument applies using an affinoid cover of $Y$ instead of $\dSp(B)$, utilizing the fact that $\Nuc(X)$ is also rigid.
\end{proof}

\subsection{6-functor formalism for nuclear categories over rigid analytic spaces}

As explained in \Cref{eq: pullback}, to every map of derived rigid analytic spaces $f\colon Y \to X$, we can associate a pullback functor at the level of nuclear \cates: 
\[
f^*\colon \Nuc(X)\to \Nuc(Y).
\]
This functor is a left adjoint, and we denote its right adjoint by $f_{\ast}$. When $X$ and $Y$ are derived affinoid spaces, $f_{\ast}$ can be identified with the forgetful functor.

\begin{theorem}\label{thm: 6 functor for nuclear modules in rigid geometry}
    The assignment that sends a derived rigid analytic space $X$ over $\Spk$ to its \cat $\Nuc(X)$ of nuclear sheaves can be extended to a $6$-functor formalism with values in stable, presentable \cates:
    \[
    \mathrm{Corr}(\cC)_{\mathrm{qcqs}, \mathrm{all}}\colon \mathrm{dRig}_{k} \to \LPr_{\mathrm{st}},
    \]
    where, for every map of quasi-compact and semi-separated derived rigid analytic spaces, we define the lower shriek functor $f_{!}$ to be the pushforward $f_{\ast}$.
\end{theorem}

In order to prove this theorem, we require the following lemma.

\begin{lemma}\label{lem: base change and projection}
    Let $\alpha \colon Y \to X$ be a map of quasi-compact, semi-separated derived rigid analytic spaces. Then the following properties hold:
    \begin{itemize}
        \item[1] (Base change) Let $\beta \colon V \to X$ be a morphism of quasi-compact and semi-separated derived rigid analytic spaces, and consider the diagram
        \[
        \begin{tikzcd}
            Y \times_X V \arrow[r, "\alpha'"] \arrow[d, "\beta'"'] 
            & V \arrow[d, "\beta"] \\
            Y \arrow[r, "\alpha"] 
            & X.
        \end{tikzcd}
        \]
        Then, for every object $M$ in $\Nuc(Y)$, the canonical map 
        \begin{equation}\label{eq: base change}
            \beta^* \alpha_* M \longrightarrow \alpha'_* \beta'^* M
        \end{equation}
        is an equivalence.
        
        \item[2] (Projection formula) For every $N$ in $\Nuc(Y)$ and $M$ in $\Nuc(X)$, the natural map 
        \begin{equation}\label{eq: projection formula}
            \alpha_{\ast}(N \otimes_{\Nuc(Y)} \alpha^{\ast}M) \longrightarrow \alpha_{\ast}N \otimes_{\Nuc(X)}M
        \end{equation} 
        is an equivalence.
        
        \item[3] The functor
        \[
        \alpha_{\ast} \colon \Nuc(Y) \to \Nuc(X)
        \]
        commutes with colimits; in particular, it admits a right adjoint.
    \end{itemize}
\end{lemma}

\begin{proof}
    We begin the proof by observing that our setup satisfies \cite[Assumption 1.3.10]{montagnani2025axiomaticapproachanalytic1affineness}, taking $\cC$ to be the \cat of derived affinoid spaces and $\widehat{\cC}$ to be the \cat $\mathrm{Shv}(\mathrm{dAdf})$.
    
    To prove (1), suppose $V$ is a derived rigid analytic variety. We can use a derived affinoid cover $\coprod_{i \in I} \dSp(A_{i})$ of $V$ and \Cref{lem: Nuc is a sheaf} to verify \Cref{eq: base change} after restriction to each derived affinoid space $\dSp(A_{i})$. In this way, we can reduce the verification of \Cref{eq: base change} to the case where $V$ is a derived affinoid space. In this case, the result follows directly from \cite[Proposition 1.3.13]{montagnani2025axiomaticapproachanalytic1affineness}.
    
    We observe that (2) is a special case of \cite[Proposition 1.3.14]{montagnani2025axiomaticapproachanalytic1affineness}, and (3) follows from \cite[Proposition 1.3.13 (2)]{montagnani2025axiomaticapproachanalytic1affineness}.
\end{proof}

\begin{proof}[Proof of \Cref{thm: 6 functor for nuclear modules in rigid geometry}]
    This follows from \Cref{lem: base change and projection} and \cite[Proposition A.5.10]{mann2022padic6functorformalismrigidanalytic}, taking $I$ to be the class of equivalences and $P$ to be the class of maps between quasi-compact and semi-separated derived rigid analytic spaces. 
\end{proof}

\begin{remark}
    We observe that in the $6$-functor formalism of \Cref{thm: 6 functor for nuclear modules in rigid geometry}, every !-able map (i.e., every map between quasi-compact and semi-separated derived rigid analytic spaces) is prim. This follows from the fact that for all such maps, the lower star and lower shriek functors coincide, by \cite[Lemma 4.5.5]{heyer20246functorformalismssmoothrepresentations}.
\end{remark}


\section{Smooth categories and compact generation in rigid analytic geometry}\label{section 4}

In this section, we study the compact generation of the \cat of nuclear modules over rigid analytic varieties. In particular, we will show that this \cat is compactly (equivalently, atomically) generated if and only if the rigid analytic variety is algebraizable (see \Cref{thm: algebraization via generator}). We then use this result to falsify \Cref{conj: congettura setting enriched}. Two important steps in the proof are a version of the GAGA theorem, as described in \cite{wang2026relativegagatheoremapplication} in the setting of nuclear \cates, and the comparison between the geometric notion of smooth rigid varieties and the categorical definition of smoothness (see \Cref{thm: main theorem smoothness}).

\subsection{Non-commutative smoothness and classical smoothness}

We begin by comparing the classical notion of smoothness in rigid analytic geometry with the notion of a smooth \cat introduced in \Cref{section 1}. 
In particular, we will prove the following theorem:

\begin{theorem}\label{thm: main theorem smoothness}
    Let $f \colon X \to \Sp(k)$ be a (classical) rigid analytic variety over $k$. Then $X$ is smooth if and only if the \cat $\Nuc(X)$ is internally smooth in $\LPr_{\Nuc(k)}$.
\end{theorem}

We observe that an analogous theorem has been proved for schemes in \cite{Orlov_2016noncommutativeschemes} and in \cite{lunts2009categoricalresolutionsingularities} for the category of perfect complexes. However, the generalization to rigid analytic varieties is not straightforward, and we cannot reproduce the same proof as in \cite{Orlov_2016noncommutativeschemes}. Indeed, the analogue of \Cref{thm: main theorem smoothness} does not hold for the \cat of perfect complexes in rigid analytic geometry. It has been shown in \cite{montagnani2025algebraizationrigidanalyticvarieties} that the notion of smoothness for the \cat of perfect complexes over a rigid analytic variety is intrinsically algebraic and does not reflect the smoothness of the variety itself. To obtain a correspondence between smooth ``categories of sheaves'' and smooth varieties, it is necessary to use a ``completed tensor product of \cates'' with the property that, for a variety $X$,
\[
\mathrm{Shv}(X\times X) \simeq \mathrm{Shv}(X)\widehat\otimes\mathrm{Shv}(X).
\]
This can be achieved, for example, in the framework of analytic geometry developed by Clausen and Scholze \cite{Clausen_Scholze_lectures}, or in the setting developed by Ben-Bassat, Kelly, and Kremnizer \cite{benbassat2024perspectivefoundationsderivedanalytic}. Here, we have chosen to explore these ideas within the setting of analytic stacks as in \cite{Clausen_Scholze_lectures}, focusing on the sheaf of nuclear \cates. However, we believe that it would not be difficult to adapt some of the ideas presented here to the setting of \cite{benbassat2024perspectivefoundationsderivedanalytic}, utilizing the $6$-functor formalism described in \cite{soor2025sixfunctorformalismquasicoherentsheaves}, or to other ``analytic sheaf theories,'' such as the $\mathscr{D}$-solid sheaves described in \cite{camargo2024analyticrhamstackrigid}.

We will prove \Cref{thm: main theorem smoothness} at the end of this section as a corollary of \Cref{prop: non commutative smoothness affine case}.

\begin{definition}\label{def:smoothAffinoid}
    An affinoid algebra $A$ over $k$ is called \emph{smooth} if, for every finite field extension $K$ of $k$, the tensor product $A\otimes_{k}K$ is a regular ring (i.e., for every maximal ideal $\mathfrak{m}$ in $A\otimes_{k}K$, the local ring $(A\otimes_{k}K)_{\mathfrak{m}}$ is regular). 
\end{definition}

We now relate this classical notion of smoothness to the notion of a smooth algebra in $\Nuc(k)$, as defined in \Cref{def: smooth algebras}.





\begin{proposition}\label{prop: non commutative smoothness affine case}
    Let $A$ be a (classical) affinoid algebra over $k$, and let $A^{\triangleright}$ be the associated condensed ring. Then $A^{\triangleright}$ is a smooth algebra in $\Nuc(k)$ in the sense of \Cref{def: smooth algebras} if and only if $A$ is smooth over $k$ in the sense of \Cref{def:smoothAffinoid}.
\end{proposition}
\begin{proof}
    We begin by assuming that $A^\triangleright$ is smooth in $\Nuc(k)$. To prove that $A$ is smooth, after replacing $k$ with a finite field extension $K$ and $A$ with its base change to $K$, it suffices to show that for every maximal ideal $\mathfrak{m}$ in $A$, the localization $A_{\mathfrak{m}}$ is regular. By Serre's criterion \cite[Lemma 11.3.3.3]{SAG}, it is enough to show that the residue field $\Tilde{K}\coloneqq A_{\mathfrak{m}}/\mathfrak{m}A_{\mathfrak{m}}$ is a perfect complex when regarded as a module over $A_{\mathfrak{m}}$. First, we observe that the following diagram is a pullback of derived affinoid spaces (where all tensor products are taken relative to $k$): 
    \begin{equation}\label{eq: pullback affinoid spaces}
    \begin{tikzcd}
        {\Sp(\Tilde{K})} && {\Sp(A)} \\
        \\
        {\Sp(A\otimes \Tilde{K})} && {\Sp(A\widehat\otimes A).}
        \arrow["x", from=1-1, to=1-3]
        \arrow["x"', from=1-1, to=3-1]
        \arrow["\Delta", from=1-3, to=3-3]
        \arrow["\mathrm{id}\times x", from=3-1, to=3-3]
    \end{tikzcd}
    \end{equation}
    This follows from the following standard computation: 
    \begin{align*}
        A \widehat{\otimes}_{A\widehat{\otimes}A}(A\widehat{\otimes}\Tilde{K}) & \simeq A \widehat{\otimes}_{A\widehat{\otimes}A}(A\widehat{\otimes}A\widehat{\otimes}_{A}\Tilde{K}) \\ & \simeq 
        (A\widehat{\otimes}_{A\widehat{\otimes}A}A\widehat{\otimes}A)\widehat{\otimes}_A \Tilde{K} \\ &
        \simeq A\widehat{\otimes}_A \Tilde{K} \simeq \Tilde{K}.
    \end{align*}
    Note that since $\Tilde{K}$ is a finite field extension, the tensor product $A\otimes_{k}\Tilde{K}$ does not need to be completed. Since $A^{\triangleright}$ is smooth, we have that $\Delta_\ast A^{\triangleright}$ is dualizable in $\Nuc(A\widehat\otimes A)$.
    
    Using the base change property of \Cref{lem: base change and projection}, we see that $\Tilde{K}$, viewed as a module over $A\otimes_{k}\Tilde{K}$, can be written as
\[
x_{\ast}\Tilde{K}^\triangleright \simeq x_\ast x^\ast (A\otimes_{k}\Tilde{K})^\triangleright \simeq (\mathrm{id} \times x)^\ast \Delta_\ast A^{\triangleright}.
\]
In particular, since $\Delta_\ast A^\triangleright$ is dualizable over $\Nuc(A\widehat\otimes A)$, it follows that $\Tilde{K}^\triangleright$ is also dualizable over $\Nuc(A\otimes_{k}\Tilde{K})$. By invoking \cite[Theorem 5.50]{andreychev2021pseudocoherent} and \cite[Lemma 5.47]{andreychev2021pseudocoherent}, we deduce that $\Tilde{K}$ is dualizable over $\Mod_{A\otimes_{k}\Tilde{K}}$. Since dualizable objects in $\Mod_{A\otimes_{k}\Tilde{K}}$ are perfect complexes and $\Tilde{K}$ is a finite field extension of $k$, we conclude that $\Tilde{K}$ is dualizable in $\Mod_{A}$.
To conclude, we observe that 
\[
\Tilde{K} \simeq \Tilde{K} \otimes_{A}A_{\mathfrak{m}}.
\]
This implies, in particular, that $\Tilde{K}$ is dualizable as a module over $A_{\mathfrak{m}}$.

To conclude the proof, we assume that $A$ is smooth, and we need to prove that $A^\triangleright$ is dualizable as a diagonal module in $\Nuc(A\widehat\otimes_{k} A)$. By \cite[Corollary 5.51.1]{andreychev2021pseudocoherent}, it is enough to show that $A$ is dualizable in $\Mod_{A\widehat\otimes_{k}A}$. The diagonal map is a closed immersion and is, moreover, a local complete intersection (lci) by \cite[Corollary 5.11]{zavyalov2025foundationalresultsadicgeometry}. Thus, we can locally assume it to be a regular immersion; using the Koszul resolution, this implies that $\Delta_{\ast}A$ is dualizable.
\end{proof}

\begin{corollary}\label{cor: non commutative smoothness global case}
    Let $X$ be a rigid analytic variety over $\Sp(k)$. Then $X$ is smooth if and only if $\Delta_{\ast}\mathscr{O}_{X}$ is a compact object in $\Nuc(X\widehat\times X)$.
\end{corollary}

\begin{proof}
    Both the smoothness of $X$ and the compactness of $\Delta_{\ast}\mathscr{O}_{X}$ are local properties on $X$ in the analytic topology (see \cite[Theorem 5.41 and Theorem 5.42]{andreychev2021pseudocoherent}). Consequently, we can reduce the proof to the affine case treated in \Cref{prop: non commutative smoothness affine case}.
\end{proof}

We conclude this subsection by proving its main result.

\begin{proof}[Proof of \Cref{thm: main theorem smoothness}]
    Since in our $6$-functor formalism $\Delta_{!}\coloneqq\Delta_{\ast}$, and the \cates considered here are rigid, we can deduce from \Cref{cor: smooth categories 6ff rigid case} that $\Nuc(X)$ is internally smooth over $\Nuc(k)$ if and only if $\Delta_\ast(\mathscr{O}_X)$ is compact. This reduces the statement to \Cref{cor: non commutative smoothness global case}.
\end{proof}

\subsection{GAGA for nuclear categories in rigid analytic geometry}

In this subsection, we will describe a GAGA theorem for nuclear \cates, utilizing some ideas presented in \cite{wang2026relativegagatheoremapplication}. 
In \emph{loc. cit.}, the author introduces two analytic stacks associated with an algebraic stack $X$, called respectively the \emph{algebraic realization} $X^{\mathrm{k-alg}}$ and the analytification $X^{\mathrm{an}}$. When $X$ is a proper scheme, the author proves a GAGA-type theorem relating the solid derived \cates of $X^{\mathrm{k-alg}}$ and $X^{\mathrm{an}}$ (see \cite[Theorem 4.1.1]{wang2026relativegagatheoremapplication}).

We begin by reviewing the \emph{algebraic realization} construction as introduced in \cite{wang2026relativegagatheoremapplication}; we refer to \emph{loc. cit.} for a comprehensive treatment. We will not review the \emph{analytification} construction described in \cite{wang2026relativegagatheoremapplication} here. Indeed, we will only consider the case where $X$ is an underived scheme of finite type. In this scenario, the analytification functor described in \cite{wang2026relativegagatheoremapplication} produces exactly the analytic stack associated to the rigid analytic variety $X^{\mathrm{an}}$ (see \cite[Lemma 3.2.9 and Lemma 3.2.12]{wang2026relativegagatheoremapplication}). For this reason, by a slight abuse of notation, we will still denote this analytic stack by $X^{\mathrm{an}}$.

If $R$ is an animated ring of finite type over $k$, we can associate to it an analytic ring $R^{\mathrm{triv}}\coloneqq (\uline{R}, \Mod_{\uline{R}}^{\mathrm{cond}})$, where $\uline{R}$ is the discrete condensed ring associated with $R$. We can also define the \emph{algebraic realization} of $R$, following \cite[Definition 3.2.1]{wang2026relativegagatheoremapplication}, to be the following pushout of analytic rings:
\[
\begin{tikzcd}
	{k^{\mathrm{triv}}} && {R^{\mathrm{triv}}} \\
	\\
	{k_{\Solid}} && {R^{\mathrm{k-alg}}}
	\arrow[from=1-1, to=1-3]
	\arrow[from=1-1, to=3-1]
	\arrow[from=1-3, to=3-3]
	\arrow[from=3-1, to=3-3]
	\arrow["\lrcorner"{anchor=center, pos=0.125, rotate=180}, draw=none, from=3-3, to=1-1]
\end{tikzcd}
\]
In particular, using \cite[Proposition 12.12]{analytic}, we can identify $R^{\mathrm{k-alg}}$ with the analytic ring given by $(\uline{R}, \cD(R^{\mathrm{k-alg}}))$, where $\uline{R}$ is the condensed derived ring defined in \Cref{def: solidification functor}, and we have an equivalence: 
\[
\cD(R^{\mathrm{k-alg}}) \simeq \Mod_{\uline{R}}(\cD(k)_{\Solid}).
\]
Moreover, since $\uline{R}$ is discrete relative to $k$ (and in particular nuclear over $k$), and $R^{\mathrm{k-alg}}$ carries the induced analytic ring structure from $k_{\Solid}$, we obtain the following equivalence via \cite[Lemma 2.9]{mikami2023fppfdescentcondensedanimatedrings}:
\[
\Nuc(R^{\mathrm{k-alg}}) \simeq \Mod_{\uline{R}}(\Nuc(k)).
\]
Similarly, if $X$ is a derived scheme of finite type, we can regard it as an analytic stack $X^{\mathrm{triv}}$, obtained by gluing analytic rings of the form $R^{\mathrm{triv}}$ in the Zariski topology. Note that since we are considering the trivial analytic ring structure on $R^{\mathrm{triv}}$, all transition maps in the gluing process to obtain $X^{\mathrm{triv}}$ are !-able. This ensures that $X^{\mathrm{triv}}$ is a well-defined analytic stack over $\mathrm{AnSpec}(k^{\mathrm{triv}})$.

We define $X^{\mathrm{k-alg}}$ via the following pullback of analytic stacks:
\[
\begin{tikzcd}
	{X^{\mathrm{k-alg}}} && {X^{\mathrm{triv}}} \\
	\\
	{\mathrm{AnSpec}(k_{\Solid})} && {\mathrm{AnSpec}(k^{\mathrm{triv}}).}
	\arrow[from=1-1, to=1-3]
	\arrow[from=1-1, to=3-1]
	\arrow[from=1-3, to=3-3]
	\arrow[from=3-1, to=3-3]
\end{tikzcd}
\] 
In particular, $X^{\mathrm{k-alg}}$ is obtained by gluing the analytic stacks associated with analytic rings of the form $R^{\mathrm{k-alg}}$. 


\begin{lemma}\label{lem: descent for discrete modules}
    The functor
    \begin{equation}\label{eq: discrete descent}
        \begin{tikzcd}
	{ \mathrm{CAlg}(\Mod_{k})} && {\mathrm{AnRing}} && \LPr
	\arrow["{(-)^{\mathrm{k-alg}}}", from=1-1, to=1-3]
	\arrow["{\Nuc(-)}", from=1-3, to=1-5]
\end{tikzcd}
    \end{equation}
    defined by sending a commutative algebra $R$ to the \cat of nuclear modules over the analytic ring $R^{\mathrm{k-alg}}$, satisfies descent for the fppf topology.
\end{lemma}
\begin{proof}
    Let $S \to R$ be an fppf cover of commutative algebras in $\Mod_{k}$. By \cite[Corollary 3.33]{mathew2016galoisgroupstablehomotopy}, $R$ is a descendable algebra in $\Mod_{S}$. Since $\uline{R}$ is nuclear over $k$ and the $k$-solidification functor in \Cref{def: solidification functor} is symmetric monoidal (see \Cref{lem: tensor product of discete modules}), it induces a functor
    \[
        \Mod_{S} \to \Nuc(S^{\mathrm{k-alg}})\simeq \Mod_{\uline{S}}(\Nuc(k)).
    \]
    
    Applying \cite[Corollary 3.21]{mathew2016galoisgroupstablehomotopy} to this functor, we deduce that $\uline{R}$ is a descendable algebra in $\Nuc(S^{\mathrm{k-alg}})\simeq \Mod_{\uline{S}}(\Nuc(k))$. Consequently, by \cite[Proposition 3.22]{mathew2016galoisgroupstablehomotopy}, we deduce that the functor in \Cref{eq: discrete descent} is a sheaf for the fppf topology.
\end{proof}

If $\{\Spec(A_{i})\}_{i \in I}$ is a Zariski cover of a separated scheme $X$ over $\Spec(k)$, we can observe that $\{\mathrm{AnSpec}(A_{i}^{\mathrm{k-alg}})\}_{i \in I}$ is a cover of $X^{\mathrm{k-alg}}$. Using \Cref{lem: descent for discrete modules}, we obtain an equivalence:
\[
\Nuc(X^{\mathrm{k-alg}}) \simeq \lim_{[n]\in\bDelta^{\mathrm{op}}}\prod_{\left\{i_1,\cdots,i_n\right\}\subseteq I}\Nuc(A_{i_1, \dots, i_{n}}^{\mathrm{k-alg}})
\]
where $\Spec(A_{i_1, \dots, i_{n}})$ denotes the intersection $\Spec(A_{i_{1}})\times_{X}\dots \times_{X}\Spec(A_{i_{n}})$.

\begin{proposition}\label{prop: algebraic GAGA}
    Let $X$ be a connective derived scheme, separated and of finite type over $k$. Then there is an equivalence of categories between $\Nuc(X^{\mathrm{k-alg}})$ and $\Qcoh(X)\otimes_{\Mod_{k}}\Nuc(k)$.
\end{proposition}
\begin{proof}
    As above, let $\{\Spec(A_{i})\}_{i \in I}$ be a cover of $X$, and let $\Spec(A_{i_1, \dots, i_{n}})$ denote the intersection $\Spec(A_{i_{1}})\times_{X}\dots \times_{X}\Spec(A_{i_{n}})$. Note that since $X$ is separated, $\Qcoh(X)$ can be written as a limit of quasi-coherent sheaves on affine schemes. In particular, using the same notation as above, we have that 
    \[
    \Qcoh(X) \simeq \lim_{[n]\in\bDelta^{\mathrm{op}}}\prod_{\left\{i_1,\cdots,i_n\right\}\subseteq I}\Mod_{A_{i_1, \dots, i_{n}}}.
    \]
    Thus, $\Qcoh(X)$ can be written as a limit of $\Mod_{A}$ for affine opens $\Spec(A)$ in $X$.
    Now, since both $\Nuc(k)$ and $\Mod_{k}$ are rigid, it follows that $\Nuc(k)$ is dualizable over $\Mod_{k}$; in particular, the functor 
    $-\otimes_{\Mod_{k}}\Nuc(k)$ commutes with limits. Therefore, we can write 
    \begin{equation}\label{eq:Kunneth limit}
        \Qcoh(X)\otimes_{\Mod_{k}}\Nuc(k) \simeq \lim_{[n]\in\bDelta^{\mathrm{op}}}\prod_{\left\{i_1,\cdots,i_n\right\}\subseteq I}\left(\Mod_{A_{i_1, \dots, i_{n}}}\otimes_{\Mod_k}\Nuc(k)\right).
    \end{equation}
    The right-hand side can be identified with 
    \[
     \lim_{[n]\in\bDelta^{\mathrm{op}}}\prod_{\left\{i_1,\cdots,i_n\right\}\subseteq I}\Mod_{\uline{A}_{i_1, \dots, i_{n}}}(\Nuc(k)).
    \]
    To conclude, it is enough to observe that $\Mod_{\uline{A}}(\Nuc(k))$ can be identified with $\Nuc(A^{\mathrm{k-alg}})$, and that $\{\mathrm{AnSpec}(A_i^{\mathrm{k-alg}})\}_{i \in I}$ defines a cover of $X^{\mathrm{k-alg}}$. Hence, using descent for $\Nuc(-)$ (\Cref{lem: descent for discrete modules}), we can identify the limit in \Cref{eq:Kunneth limit} with $\Nuc(X^{\mathrm{k-alg}})$.
\end{proof}

\begin{corollary}\label{cor: GAGA for nuclear categories}
    Let $X$ be a proper scheme of finite type, quasi-compact and separated, over $\Spec(k)$, and let $X^{\mathrm{an}}$ be its analytification as a rigid analytic variety. Then we have an equivalence 
    \[
    \Nuc(X^{\mathrm{an}}) \simeq \Qcoh(X) \otimes_{\Mod_{k}} \Nuc(k).
    \]
\end{corollary}
\begin{proof}
    This is a consequence of \cite[Theorem 4.1.1]{wang2026relativegagatheoremapplication} and \Cref{prop: algebraic GAGA}. Indeed, as explained in the proof of \cite[Theorem 4.1.1]{wang2026relativegagatheoremapplication}, the pullback along the inclusion of the categorical locales associated with $X^{\mathrm{an}}$ and $X^{\mathrm{k-alg}}$ yields an equivalence of solid derived \cates:
    \[
    \cD(X^{\mathrm{k-alg}})_\Solid \simeq \cD(X^{\mathrm{an}})_\Solid.
    \]
    Since this pullback is a symmetric monoidal functor, it preserves the nuclear subcategories (see \cite[Theorem 2.4 (c)]{meyer2025qhodgecomplexesrefinedoperatornametc}). We can therefore conclude by applying \Cref{prop: algebraic GAGA}.
\end{proof}

\subsection{Compact generation of nuclear categories in rigid analytic geometry}

\begin{lemma}\label{lem: perfect complexes over rigid varieties}
    Let $X$ be a rigid analytic variety over $k$, and let $\cF$ be an object in $\Nuc(X)$. Then the following properties are equivalent:
    \begin{enumerate}
        \item $\cF$ is a dualizable object in $\Nuc(X)$.
        \item $\cF$ is a compact object in $\Nuc(X)$.
        \item $\cF$ is a $\Nuc(k)$-compact (equivalently, $\Nuc(k)$-atomic) object in $\Nuc(X)$.
        \item $\cF$ is a perfect complex over $X$.
    \end{enumerate}
\end{lemma}
\begin{proof} 
    The equivalence between (2) and (3) follows from \Cref{prop: compact objects over rigid categories}.

    We also observe that all these properties are local on $X$ (see \cite{andreychev2021pseudocoherent}), so we may assume $X$ to be an affinoid space. Since $\Nuc(k)$ is rigid, and locally we can assume $\Nuc(X)$ to be of the form $\Mod_{A}(\Nuc(k))$, the equivalence between (1) and (3) follows from \Cref{cor: dualizable if and only if V-atomic}. 
    
    Moreover, we note that by \cite[Corollary 5.51.1]{andreychev2021pseudocoherent}, we can identify dualizable objects in $\cD(X)_\Solid$ with perfect complexes over $X$. To conclude, we observe that by \cite[Lemma 5.45 and Theorem 5.50]{andreychev2021pseudocoherent}, we can also identify dualizable objects in $\cD(X)_\Solid$ with dualizable objects in $\Nuc(X)$. 
\end{proof}

\begin{theorem}\label{thm: algebraization via generator}
    Let $p\colon X \to \Sp(k)$ be a smooth, proper, quasi-compact, and separated rigid analytic variety over $\Sp(k)$. Then $X$ is algebraizable if and only if $\Nuc(X)$ admits a compact (equivalently, $\Nuc(k)$-atomic) generator.  
\end{theorem}
\begin{proof} 
    The strategy of the proof is to show that $\Perf(X)$ is a smooth, proper, idempotent-complete $k$-linear \cat, and then apply \cite[Theorem 3.1]{montagnani2025algebraizationrigidanalyticvarieties}.

    Suppose there exists an object $\cF$ which is a $\Nuc(k)$-compact generator of $\Nuc(X)$. Since $\cF$ is a $\Nuc(k)$-compact generator, we can apply the Barr--Beck--Lurie theorem (\cite[Proposition 4.8.5.8]{HA}) to the functor
    \[
    \hom^{\Nuc(k)}(\cF,-) \colon \Nuc(X) \to \Nuc(k).
    \]
    In this way, we can identify $\Nuc(X)$ with $\LMod_{A}(\Nuc(k))$, where $A$ is the $\mathbb{E}_{1}$-algebra of endomorphisms of the $\Nuc(k)$-compact generator $\cF$. Explicitly,
    \[
    A \coloneqq p_{\ast}\uline{\Hom}_{\Nuc(X)}(\cF,\cF) \simeq \hom^{\Nuc(k)}(\cF, \cF) 
    \]
    as an $\mathbb{E}_{1}$-algebra in $\Nuc(k)$.

    Since dualizable objects in $\Nuc(X)$ (respectively, in $\Nuc(k)$) correspond to perfect complexes on $X$ (respectively, on $\Sp(k)$) by \Cref{lem: perfect complexes over rigid varieties}, and since $X$ is proper, \cite[Proposition 7.8]{porta2018derivedhomspacesrigid} implies that $A$ is a perfect complex of $k$-modules. In particular, it is a dualizable object in $\Nuc(k)$ and thus a proper algebra in $\Nuc(k)$. Because $X$ is smooth, \Cref{cor: non commutative smoothness global case} implies that $\Nuc(X)$ is smooth over $\Nuc(k)$. Therefore, by \Cref{prop: smooth cates vs smooth algebras}, $A$ is a smooth algebra over $\Nuc(k)$. Using \cite[Proposition 4.6.4.4 and Proposition 4.6.4.12]{HA}, we obtain
    \[
    \Perf(X)\simeq \Nuc(X)^{\mathrm{dual}} \simeq \LMod_{A}(\Nuc(k))^{\mathrm{dual}}\simeq \LMod_{A}(\Perf(k)) \simeq \Perf(A).
    \]

    Now, to prove that $\Perf(X)$ is a smooth and proper idempotent-complete $k$-linear \cat, it suffices to show that $A$ is a smooth and proper $k$-algebra. We observe that, since $A$ is a perfect complex over $k$, the completed tensor product $A\widehat\otimes_{k}A$ coincides with the uncompleted tensor product $A\otimes_{k}A$. This implies that $A$ is a smooth algebra over $\Mod_{k}$. Furthermore, since $A$ is a perfect complex of $k$-modules, it is dualizable in $\Mod_{k}$, meaning it is also a proper algebra in $\Mod_{k}$. Hence using \cite[Theorem 3.1]{montagnani2025algebraizationrigidanalyticvarieties} we deduce that $X$ is algebraizable.

    The converse follows from \Cref{cor: GAGA for nuclear categories}. Indeed, we have an equivalence
    \[
    \Nuc(X^{\mathrm{an}})\simeq \Qcoh(X)\otimes_{\Mod_{k}}\Nuc(k).
    \]
    In particular, using \Cref{cor: non commutative algebra in alg geo}, we deduce the existence of a discrete, smooth, and proper algebra $A$ in $\Mod_{k}$ such that
    \[
    \Qcoh(X) \simeq \LMod_{A}.
    \]
    This implies that 
    \[
    \Nuc(X) \simeq \LMod_{\uline{A}}(\Nuc(k)).
    \]
    Since $A$ is a perfect complex over $k$, the same argument explained above shows that $\uline{A}$ is a smooth and proper algebra in $\Nuc(k)$.
\end{proof}

\begin{corollary}\label{cor: counterxample}
    Let $p\colon X \to \Sp(k)$ be a smooth, proper, quasi-compact, separated and non-algebraizable rigid analytic variety over $\Sp(k)$. Then $\Nuc(X)$ is internally smooth over $\Nuc(k)$, but it does not have a compact (or, equivalently, a $\Nuc(k)$-atomic) generator. 
\end{corollary}

\begin{proof}
    In this case, \Cref{cor: non commutative smoothness global case} shows that $\Nuc(X)$ is smooth over $\Nuc(k)$, while \Cref{thm: algebraization via generator} demonstrates that $\Nuc(X)$ does not admit a compact generator.
\end{proof}

\begin{remark}
     Using \Cref{cor: counterxample} we are able to produce an example of a category that is internally smooth over a rigid base but fails to have a compact generator. It suffices to consider the category of nuclear modules over a non-algebraizable, smooth, and proper rigid analytic variety, such as the $p$-adic Hopf surface studied in \cite{mustafin1977padic, voskuil1991nonarchimedean}.
     This provides a counterexample to \Cref{conj: congettura setting enriched}, illustrating that Stefanich's theorem (\cite{stefanich2023classification}) cannot be extended to arbitrary rigid symmetric monoidal bases.
\end{remark}

\begin{remark}
    Our previous theorem implies that it is not, in general, possible to glue local, relatively compact generators of a sheaf of categories in the analytic topology to obtain a global compact generator. Locally in the analytic topology over $X$, the sheaf $\Nuc(-)$ is of the form $\Mod_{A}(\Nuc(k))$ for affinoid algebras $A$. However, \Cref{thm: algebraization via generator} demonstrates that, globally, such a generator algebra does not necessarily exist.
    
    This property of gluing local compact generators is one of the foundational features of Non-Commutative Algebraic Geometry, and it highlights a key distinction between algebraic and analytic topologies. This gluing procedure for compact generators has been shown to hold for various topologies utilized in algebraic geometry; see, for example, \cite{toen2007moduliobjectsdgcategories} and \cite{bondal2002generatorsrepresentabilityfunctorscommutative}.
\end{remark}
\appendix
\section{Smooth and proper categories in derived algebraic geometry}\label{Appendix 1}

In this appendix, we will always work over a discrete field $k$ of characteristic $0$, and we will employ the foundations for derived geometry as described in \cite{SAG} and \cite{hag}. We emphasize that the results presented here are known to experts; however, we provide an alternative proof based on the six-functor formalism and the theory of smooth \cates as developed in \Cref{section 1}. In particular, we will prove the following theorem.

\begin{theorem}\label{thm: smooth and proper cat in alg geo}
    Let $f \colon X \to Y$ be a morphism of finite Tor-amplitude between quasi-compact, semi-separated derived schemes of finite type. Then the following hold:
    \begin{enumerate}
        \item The map $f$ is smooth if and only if $\mathrm{QCoh}(X)$ is smooth over $\Qcoh(Y)$.
        \item If $\mathrm{QCoh}(X)$ is proper over $\Qcoh(Y)$, then $f$ is universally closed (i.e., for every derived scheme $W$ and for every map $W \to Y$, the induced map of topological spaces $|X\times_{Y}W| \to |W|$ is closed).
        \item If $f$ is proper, then $\mathrm{QCoh}(X)$ is proper over $\Qcoh(Y)$.
    \end{enumerate}
\end{theorem}

\begin{remark}
    In particular, under the hypotheses of \Cref{thm: smooth and proper cat in alg geo}, if $f$ is separated, then $f$ is proper if and only if $\Qcoh(X)$ is proper over $\Qcoh(Y)$. Moreover, as explained in \cite[Theorem 11.4.2.1]{SAG}, if we assume $\Qcoh(X)$ to be smooth, then $X$ is automatically separated (since the diagonal is a closed immersion; see \cite[Lemma 11.3.6.2]{SAG}). This implies that, under the hypotheses of \Cref{thm: smooth and proper cat in alg geo}, $X$ is smooth and proper if and only if $\Qcoh(X)$ is smooth and proper over $\Qcoh(Y)$.
\end{remark}

As a corollary, we immediately deduce the following, which is the algebraic geometry analogue of \Cref{thm: main theorem smoothness}.

\begin{corollary}
    Let $X \to \Spec(k)$ be a quasi-compact and separated underived scheme of finite type over $k$. Then $X$ is smooth (respectively, proper) if and only if $\Qcoh(X)$ is smooth (respectively, proper).
\end{corollary}

When $f$ is the structure map $X \to \Spec(k)$ from a (classical) scheme to a field, a similar statement has been proved in \cite[Proposition 3.30]{Orlov_2016noncommutativeschemes} regarding properness, and in \cite[Proposition 3.13]{lunts2009categoricalresolutionsingularities} regarding smoothness, by considering the DG category of perfect complexes over $X$. Our corollary is also closely related to \cite[Theorem 11.4.2.1]{SAG} (see also \cite[Proposition 11.1.4.3 and Theorem 11.3.6.1]{SAG} for results similar to those obtained in \Cref{thm: smooth and proper cat in alg geo}).

However, the techniques we employ in the proof of \Cref{thm: smooth and proper cat in alg geo} are different and rely on a specific six-functor formalism for quasi-compact and semi-separated schemes.

We begin by reviewing some classical definitions to clarify our statements.

We consider derived schemes to be those derived stacks that admit a Zariski cover by affine derived schemes of finite type over $k$, glued together along Zariski open immersions.

We recall the following definitions of quasi-compact, semi-separated, and separated maps.

\begin{definition}
    Let $f\colon X \to Y$ be a map of derived schemes of finite type over $k$. 
    \begin{enumerate}
        \item We say that $f$ is quasi-compact if, for every map $\Spec(A) \to Y$, the pullback $\Spec(A)\times_{Y}X$ admits a finite cover by affine schemes.
        \item We say that $f$ is \emph{semi-separated} if the diagonal $\Delta_f$ is affine.
        \item We say that $f$ is \emph{separated} if the diagonal $\Delta_f$ is a closed immersion. 
        \item We say that a derived stack $X$ over $\Spec(k)$ is quasi-compact (respectively, semi-separated) if the canonical map $X \to \Spec(k)$ is quasi-compact (respectively, semi-separated). In particular, in this case, $X$ admits a finite atlas of affine derived schemes whose intersections are again affine.
    \end{enumerate}
\end{definition}

\begin{theorem}[{\cite[Theorem 8.11]{scholze2025sixfunctorformalisms}}]
\label{thm:6-functor-for-qc-qs-schemes}
    There exists a symmetric monoidal six-functor formalism on derived schemes
    \[
    \Qcoh(-)\colon \mathrm{Corr(Sch_{k})}_{\mathrm{qcss, all}} \to \LPr_{\mathrm{st}}
    \]
    with the following properties. A derived stack $X$ is assigned to the \cat $\mathrm{QCoh}(X)$, and for every map of schemes $f\colon X \to Y$, we consider the usual pullback and pushforward functors. We take the class of !-able maps to be the maps between quasi-compact and semi-separated derived schemes, for which we define $f_{!} \coloneqq f_{\ast}$.
\end{theorem}

We remark here that this six-functor formalism is not the one commonly used in algebraic geometry; in particular, our upper shriek functors are not the ones used in Grothendieck--Serre duality. However, for our purposes, this six-functor formalism is sufficient.

\begin{proof}
    This statement follows from \cite[Proposition 1.3.13]{montagnani2025axiomaticapproachanalytic1affineness} and \cite[Proposition A.5.10]{mann2022padic6functorformalismrigidanalytic} by taking $I$ to be the class of identity maps and $P$ to be the class of maps between quasi-compact, semi-separated derived schemes. The fact that this six-functor formalism is symmetric monoidal easily follows from the fact that, for a quasi-compact and semi-separated derived scheme $X$, the \cat $\Qcoh(X)$ is rigid (see \cite[Proposition 1.3.17]{montagnani2025axiomaticapproachanalytic1affineness} for a proof of this fact).
\end{proof}

\begin{proof}[Proof of \Cref{thm: smooth and proper cat in alg geo}]
We begin by proving (1). We consider the coevaluation functor obtained as in \Cref{eq: coevaluation 6 functor formalism}:
\[
\begin{tikzcd}
	& {\mathrm{coev}\colon\Qcoh(Y)} && {\Qcoh(X)} && {\Qcoh(X\times_{Y} X).}
	\arrow["{f^{\ast}}"', from=1-2, to=1-4]
	\arrow["{\Delta_{f_!}}"', from=1-4, to=1-6]
\end{tikzcd}
\]

By \Cref{lem: smooth and proper over a rigid base}, we deduce that $\Qcoh(X)$ is smooth over $\Qcoh(Y)$ if and only if $\Delta_\ast \cO_X$ is compact in $\Qcoh(X \times_Y X)$. Since both conditions are local with respect to the Zariski topology, we may assume for the remainder of the proof that $X$ and $Y$ are affine. Suppose now that $\Qcoh(X)$ is smooth. To prove that $f$ is also smooth, we may assume, by virtue of \cite[Proposition 11.2.3.6]{SAG}, that $Y \simeq \Spec(k)$ for some field $k$. In this setting, the claim follows by adapting the argument of \Cref{prop: non commutative smoothness affine case}. Indeed, that proof relies on the base-change property and on the fact that compact and dualizable objects coincide in $\Nuc(A)$; both of these properties hold true in the setting considered here. Conversely, if $f$ is smooth, then the diagonal morphism $\Delta_f$ is a quasi-smooth (i.e., local complete intersection) closed immersion. These properties ensure that the pushforward along $\Delta_f$ preserves perfect complexes (see, e.g., \cite{toen2012properlocalcompleteintersection}).

To prove (3), we assume $f$ is proper, and we need to prove that the evaluation functor
\[
\begin{tikzcd}
	{\ev\colon\Qcoh(X\otimes_{Y} X)} && {\Qcoh(X)} && {\Qcoh(Y)}
	\arrow["{\Delta^*}"', from=1-1, to=1-3]
	\arrow["{f_!}"', from=1-3, to=1-5]
\end{tikzcd}
\]
has a continuous right adjoint. Since $\Delta_\ast \simeq \Delta_!$, it is enough to show that $f^{!}$ is continuous, which follows from \cite[Theorem 8.13]{scholze2025sixfunctorformalisms}. 

To prove statement (2), let $W$ be a derived scheme over $k$ and consider a map $W \to Y$, we must prove that the induced map $X\times_Y W \to W$ is closed. As also explained in \cite[\href{https://stacks.math.columbia.edu/tag/0CL3}{Tag 0CL3}]{stacks-project} by covering $W$ with affine derived schemes, we can reduce to prove the statement when $W$ is affine. Now in this setting the \cates of quasi-coherent sheaves are rigid, this easily implies the following K\"unneth formula (see, for example, the proof of \Cref{prop: kunneth for nuclear moduels}, or \cite{benzvi2010integraltransformsdrinfeldcenters}):
\[
\Qcoh(X\times_Y W) \simeq \Qcoh(X)\otimes_{\Qcoh(Y)} \Qcoh(W).
\]

Since proper \cates are stable under base change (see \cite[Proposition 11.1.1.5]{SAG}), we deduce that $\Qcoh(X\times_Y W)$ is proper over $\Qcoh(W)$. The statement now follows from \cite[Lemma 11.1.4.6]{SAG}.
\end{proof}

As a corollary of \Cref{thm: smooth and proper cat in alg geo}, we obtain the following.

\begin{corollary}[{\cite{bondal2002generatorsrepresentabilityfunctorscommutative, toen2011derivedazumayaalgebrasgenerators}}]\label{cor: non commutative algebra in alg geo}
    Let $p\colon X \to \Spec(k)$ be a smooth, proper, quasi-compact, and semi-separated scheme over $\Spec(k)$. Then there is a smooth and proper $\mathbb{E}_{1}$-algebra $A$ in $\Mod_{k}$ such that $\Qcoh(X)$ is equivalent to $\mathrm{RMod}_{A}$.
\end{corollary}

\begin{proof}
    It is a standard result that $\Qcoh(X)$ is compactly generated by its full subcategory of perfect complexes (see, for example, \cite[Theorem 8.3]{scholze2025sixfunctorformalisms}). By \Cref{thm: smooth and proper cat in alg geo}, $\Qcoh(X)$ is also smooth; using \cite[Proposition 11.3.2.4]{SAG}, we can deduce that it admits a single compact generator, say $\cG$. 
    
    Using the same argument as in \Cref{prop: compact over Sp}, we can see that $\cG$ is also a compact generator relative to $\Mod_{k}$. To show the existence of the algebra $A$, we will use the Barr--Beck--Lurie theorem (\cite[Proposition 4.8.5.8]{HA}). To apply \textit{loc.\ cit.}, we only need to show that for every $M \in \Mod_{k}$ and $\cF \in \Qcoh(X)$, the natural map
    \begin{equation}\label{eq: equivalence for Bar-Back}
    \hom^{\Mod_{k}}(\cG, \cF)\otimes_{k} M \to \hom^{\Mod_{k}}(\cG, p^{\ast}(M)\otimes_{\cO_{X}} \cF) 
    \end{equation}
    is an equivalence. To prove this, we consider $\cX$, the stable full subcategory of $\Mod_{k}$ spanned by all objects $M \in \Mod_{k}$ such that the map in \Cref{eq: equivalence for Bar-Back} is an equivalence. Obviously, $k$ is in $\cX$, and $\cX$ is closed under colimits since $\cG$ is compact relative to $\Mod_{k}$. This implies that $\cX$ is the entirety of $\Mod_{k}$, and \Cref{eq: equivalence for Bar-Back} is an equivalence for all $M$. Defining $A$ as the $\mathbb{E}_{1}$-algebra
    \[
    A \coloneq p_{\ast}\uline{\Hom}(\cG, \cG) \simeq \hom^{\Mod_{k}}(\cG,\cG),
    \]
    we thus obtain an equivalence between $\Qcoh(X)$ and $\RMod_{A}$. Finally, we need to show that $A$ is a smooth and proper algebra. The fact that $A$ is smooth follows from \Cref{prop: smooth cates vs smooth algebras} and \Cref{thm: smooth and proper cat in alg geo}. Moreover, $A$ is proper since $X$ is proper, and in this case $p_{\ast}$ preserves perfect complexes (see \cite[Theorem 8.10]{scholze2025sixfunctorformalisms}).
\end{proof}

We conclude by recalling that, for a quasi-compact and quasi-separated scheme, the existence of a compact generator was proved in \cite{bondal2002generatorsrepresentabilityfunctorscommutative}, and in a more general setting in \cite{toen2011derivedazumayaalgebrasgenerators}.

\section{Smooth and proper categories in analytic stacks}\label{Appendix 2}

In this section, we apply the results obtained in \Cref{section 1} to provide examples of smooth and proper \cates for analytic stacks, as introduced by Clausen and Scholze. We employ the theory of analytic stacks developed in \cite{Clausen_Scholze_lectures}. For a comprehensive overview, we refer the reader to \cite{kesting2025categoricalkunnethformulasanalytic}. In what follows, we consider the \cat of analytic stacks defined as the \cat of sheaves with values in $\mathrm{Anima}$ with respect to the $\mathscr{D}$-topology on the \cat of analytic rings.

In this appendix, we focus on analytic stacks arising from rigid analytic geometry. However, it should be straightforward to adapt these ideas to other contexts, for example, by viewing schemes as analytic stacks endowed with the trivial analytic ring structure, or complex manifolds equipped with the liquid (or gaseous) analytic ring structure.

In this section, we fix a non-Archimedean field $k$ with a pseudo-uniformizer $\omega$ and a ring of power-bounded elements $k^\circ$. and we will prove the following result.

\begin{proposition}\label{prop: smooth D-solid cat rigid geometry}
    Let $f\colon X \to \Sp(k)$ be a smooth, proper, quasi-compact, and separated rigid analytic variety. Then, considering the induced map of analytic stacks $f_\Solid \colon X_\Solid \to \mathrm{AnSpec}(k_\Solid)$, the \cat $\mathscr{D}(X)_\Solid$ is internally smooth and proper over $\mathscr{D}(k)_\Solid$. 
\end{proposition}

\begin{remark}
    We can interpret the preceding proposition as stating that the coevaluation functor
    \begin{equation*}
        \begin{tikzcd}
	{\coev\colon\cD(k)_\Solid} && {\cD(X)} && {\cD(X\widehat{\times}X)}
	\arrow["{f_\Solid^*}"', from=1-1, to=1-3]
	\arrow["{\Delta_{f_{\Solid},!}}"', from=1-3, to=1-5]
\end{tikzcd}
    \end{equation*}
    admits a continuous and linear right adjoint. By the results in \Cref{section 1}, we deduce that this implies $\Delta_{f_\Solid, !}(\boldone_{X})$ is atomic relative to $\mathscr{D}(k)_\Solid$. However, the precise relationship between $\mathscr{D}(k)_\Solid$-atomic objects in $\mathscr{D}(X\widehat{\times}X)$ and perfect complexes over $X\widehat{\times}X$ remains unclear to the author. This prevents us from deducing the smoothness of the variety from the smoothness of the \cat $\mathscr{D}(X)_\Solid$ and the analogous results established in \Cref{section 4}.
\end{remark}

We begin by explaining how to construct a functor from rigid analytic varieties to analytic stacks. For this reason, throughout this appendix, we restrict our attention to Huber pairs weakly of finite type. Related ideas have also been discussed in \cite{mikami2026finitenessdualitycohomologyvarphigammamodules} and \cite{andreychev2021pseudocoherent}.

We first recall the following result.

\begin{lemma}[{\cite[Lemma 2.14]{porta2016higheranalyticstacksgaga}}]\label{lem: colimit preserving functor sheaves}
    Let $\mathscr{C}$ and $\mathscr{X}$ be $\infty$-sites, and let $F\colon \mathscr{C} \to \mathscr{X}$ be a functor that preserves coverings and finite limits. Then $F$ can be extended to a functor   $\mathrm{Shv}(\mathscr{C}) \to \mathrm{Shv}(\mathscr{X})$ that preserves finite limits and colimits. 
\end{lemma}

To every analytic Huber pair $(A,A^+)$ we associate an analytic ring $(A,A^+)_\Solid$, see \cite[Theorem 3.28]{andreychev2021pseudocoherent}.

\begin{proposition}
    The functor from the category of analytic Huber pairs weakly of finite type over $k$ to analytic stacks
    \[
    (-)_\Solid \colon\mathrm{AnHuber_{\text{k}}^{wft}} \to \mathrm{AnRing}_{k_\Solid}
    \]
    defined by sending $(A,A^+)$ to $\mathrm{AnSpec}(A,A^+)_\Solid$ is a fully faithful embedding that commutes with finite limits. Moreover, it sends covers in the analytic topology to $\mathscr{D}$-covers. 
\end{proposition}
\begin{proof}
    The fact that the embedding is fully faithful was proved in \cite[Theorem 3.34]{andreychev2021pseudocoherent}. The final assertion follows from \cite[Lemma 3.2.9]{camargo2024analyticrhamstackrigid}. Since the functor $(-)_\Solid$ preserves the initial object it is enough to show that the functor $(-)_\Solid$ commutes with tensor products. Indeed, given a pushout diagram of Huber pairs weakly of finite type
    \[
    \begin{tikzcd}
	{(A,A^+)} && {(B,B^+)} \\
	\\
	{(C,C^+)} && {(D,D^+),}
	\arrow[from=1-1, to=1-3]
	\arrow[from=1-1, to=3-1]
	\arrow[from=1-3, to=3-3]
	\arrow[from=3-1, to=3-3]
\end{tikzcd}
    \]
    we can identify the ring $D$ with the completed tensor product of Tate algebras $B\widehat{\otimes}_A C$, and $D^+$ with the minimal open and integrally closed subring containing the images of $B^+$ and $C^{+}$. The analogous tensor product in the \cat of analytic rings yields an analytic ring $E\coloneqq(E^{\triangleright},\mathscr{D}(E))$, where $E^{\triangleright}$ coincides with the condensed ring associated to the completed tensor product $B\widehat{\otimes}_A C$ (see \Cref{rmk: completed tensor}). By \cite[Proposition 12.12]{analytic}, the analytic ring structure on $E$ is obtained by completing the analytic ring structure induced by $(B,B^+)_\Solid$ and $(C,C^+)_\Solid$. Since $E^\triangleright$ is an affinoid algebra, it is nuclear over $(B,B^+)_\Solid$ and $(C,C^+)_\Solid$; in particular, the induced analytic ring structure is already complete. Applying \cite[Proposition 3.32]{andreychev2021pseudocoherent}, we deduce that $(D,D^+)_\Solid$ coincides with $E$.
\end{proof}

Applying \Cref{lem: colimit preserving functor sheaves}, we obtain the following colimit-preserving functor:
    \[
    \begin{tikzcd}
	{\mathrm{AnHuber_{\text{k}}^{wft}}} && {\mathrm{AnStack}}_{k_\Solid} \\
	{\mathrm{Shv}(\mathrm{AnHuber^{wft}}_k)}
	\arrow["{(-)_\Solid}", from=1-1, to=1-3]
	\arrow[hook, from=1-1, to=2-1]
	\arrow["{(-)_\Solid}"', from=2-1, to=1-3]
\end{tikzcd}
    \]
In particular, by composing this functor with the natural map $\mathrm{AnAdic} \to \mathrm{Shv}_{\mathrm{an}}(\mathrm{AnHuber})$, we obtain the required functor:
\[
    \mathrm{RigVar} \to \mathrm{AnStacks}.
\]
Explicitly, for a rigid analytic variety $X$, we can consider the associated analytic stack, defined as: 
\begin{equation}\label{eq: rigid varieties as an stacks}
X_{\Solid} \coloneqq \varinjlim_{\mathrm{Sp}(A) \hookrightarrow X} \mathrm{AnSpec}(A,A^\circ)_{\Solid}.    
\end{equation}

For a rigid analytic variety $X$, we denote by $\mathscr{D}(X)_\Solid$ the \cat of ``analytic quasi-coherent sheaves'' on $X$, defined via descent using \Cref{eq: rigid varieties as an stacks}. 
Moreover, if $f\colon \Sp(A) \to X$ is a cover, we can identify the analytic stack $X_\Solid$ with the \v{C}ech nerve of $f_\Solid$. 

\begin{lemma}\label{lem: prim cover proper rigid variety}
    Let $f \colon X \to \mathrm{Spa}(k,k^\circ)$ be a proper and quasi-compact rigid analytic variety. Then there exists a descendable $\mathscr{D}$-prim cover of analytic stacks of the form $f_{k_\Solid}\colon \mathrm{AnSpec}(A,k^\circ)_\Solid \to X_\Solid$.
\end{lemma}
\begin{proof}
We observe that since $X$ is quasi-compact, there exists a cover $f\colon \Sp(A) \to X$ consisting of the disjoint union of finitely many open affinoids in $X$. Because $X$ is separated, the \v{C}ech nerve of $f$ is $n$-truncated for some $n$, taking the form:
\[
X \simeq \colim_{[m]\in\bDelta_\leq n}(\mathrm{Sp}(A) \stackrev{3} \mathrm{Sp}(A_{2}) \stackrev{5} \cdots)
\]
where we denote by $\Sp(A_n)$ the $n$-fold intersection $\Sp(A)\times_{X}\dots \times_{X}\Sp(A)$. Since $X$ is proper and separated, this colimit can be equivalently described in the category of adic spaces using the analogous diagram with the adic compactifications $\Spa(A_n,k^\circ)$. Applying \Cref{lem: colimit preserving functor sheaves}, we obtain an effective epimorphism: 
\[
f_{k_\Solid}\colon \mathrm{AnSpec}(A,k^\circ)_\Solid \to X_\Solid.
\]
This map satisfies the assumptions of \Cref{prop: D-prim cover}; it follows that $f_{k_\Solid}$ is a descendable $\mathscr{D}$-prim cover.
\end{proof}

\begin{proof}[Proof of \Cref{prop: smooth D-solid cat rigid geometry}]
    The proof relies on \Cref{prop: smooth and proper cat in 6 ff}. To apply loc.\ cit., we must verify that $f_\Solid$ is $\mathscr{D}$-prim and $\mathscr{D}$-suave. The fact that $f_\Solid$ is $\mathscr{D}$-suave follows from \cite[Theorem 1.62]{mikami2026finitenessdualitycohomologyvarphigammamodules}. Furthermore, using \Cref{lem: prim cover proper rigid variety} and \Cref{cor: D-prim map}, we deduce that $f_\Solid$ is indeed $\mathscr{D}$-prim.
\end{proof}

\printbibliography

\Addresses

\end{document}